 \documentclass[final,1p,times,twocolumn]{elsarticle}
\usepackage{amssymb}
 \usepackage{amsthm}

\usepackage{latexsym}
\usepackage{amsmath}
\usepackage{epsfig}
\usepackage{verbatim}
\usepackage{makeidx}
\usepackage{graphicx}
\usepackage[notref,notcite]{showkeys}
\newcommand{\detail}[1]{ \par \smallskip \noindent {\bf xxx \newline [Detail\ }{#1}
\hfill{\bf ]\newline xxx}\par \smallskip \noindent \hspace{-4pt}}
\renewcommand{\detail}[1]{}

\newcommand{\IP}{\mathbb P}

\newcommand{\IE}{\mathbb E}
\newcommand{\IR}{\mathbb R}

\newcommand{\IZ}{\mathbb Z}
\newcommand{\IN}{\mathbb N}

\newcommand{\SG}E

\newcommand{\beq}{\begin{equation}}
\newcommand{\eeq}{\end{equation}}
\newcommand{\beqn}{\begin{eqnarray}}
\newcommand{\eeqn}{\end{eqnarray}}
\newcommand{\bc}{\begin{center}}
\newcommand{\ec}{\end{center}}
\newcommand{\bi}{\begin{itemize}}
\newcommand{\ei}{\end{itemize}}

\newcommand*{\abs}[1]{\left\vert #1 \right\vert} % Absolutbetrag, Mächtigkeit
\newcommand*{\norm}[1]{\left\Vert #1 \right\Vert} % Norm
\newcommand*{\set}[1]{\left \{ #1 \right \} } % Mengenklammern
\newcommand*{\Mid}{\: \middle | \:} % Restriktionen in Mengen ("... für die gilt ...")
\newcommand*{\diff}{\textup d} % wie in df/dx oder für den Gebrauch in Inegralen
\newcommand*{\F}{\mathcal F}

\newcommand*{\Ps}{\mathcal P}

\newcommand*{\defeq}{\mathrel{\mathop{:}}=} % definitorisches Gleichheitszeichen
\newcommand*{\N}{\mathbb N} % die natürlichen Zahlen
\newcommand*{\Z}{\mathbb Z} % die ganzen Zahlen
\newcommand*{\R}{\mathbb R} % die reellen Zahlen

\newcommand*{\limty}[1][n]{\lim_{#1 \rightarrow \infty}}

\newcommand*{\konv}[1][n]{\overset{#1 \rightarrow \infty}{\longrightarrow}} % konvergiert
\newcommand*{\konvD}{\Rightarrow} % konvergiert in Verteilung
\renewcommand*{\epsilon}{\varepsilon}

\newtheorem{thm}{Theorem}[section]
\newtheorem{propn}[thm]{Proposition}
\newtheorem{lemma}[thm]{Lemma}

\newtheorem{defn}[thm]{Definition}

\newtheorem{cor}[thm]{Corollary}

\journal{-}%Theoretical Population Biology

\begin{document}

\begin{frontmatter}

%% Title, authors and addresses

%% use the tnoteref command within \title for footnotes;
%% use the tnotetext command for the associated footnote;
%% use the fnref command within \author or \address for footnotes;
%% use the fntext command for the associated footnote;
%% use the corref command within \author for corresponding author footnotes;
%% use the cortext command for the associated footnote;
%% use the ead command for the email address,
%% and the form \ead[url] for the home page:
%%
%% \title{Title\tnoteref{label1}}
%% \tnotetext[label1]{}
%% \author{Name\corref{cor1}\fnref{label2}}
%% \ead{email address}
%% \ead[url]{home page}
%% \fntext[label2]{}
%% \cortext[cor1]{}
%% \address{Address\fnref{label3}}
%% \fntext[label3]{}

\title{On spatial coalescents with multiple mergers in two dimensions}
%Alternative title: Coalescents for spatially structured populations with larger variance in the offspring number

%% use optional labels to link authors explicitly to addresses:
%% \author[label1,label2]{<author name>}
%% \address[label1]{<address>}
%% \address[label2]{<address>}

\author{Benjamin Heuer and Anja Sturm\fnref{corresponding}}
 \fntext[corresponding]{Corresponding author: Email asturm@math.uni-goettingen.de}

\address{Institute for Mathematical Stochastics\\
Georg-August-Universit\"at G\"ottingen \\
Goldschmidtstr. 7\\
37077 G\"ottingen, Germany}

\begin{abstract}
We consider the genealogy of a sample of individuals taken  from a spatially structured population when the variance of the offspring distribution is relatively large. The space is structured into discrete sites of a graph $G.$ If the population size at each site is large, spatial coalescents with multiple mergers, so called spatial $\Lambda$-coalescents, for which ancestral lines migrate in space and coalesce according to some $\Lambda$-coalescent mechanism, are shown to be appropriate approximations to the genealogy of a sample of individuals.

We then consider as the graph $G$ the two dimensional torus with side length $2L+1$ and show that as $L$ tends to infinity, and time is rescaled appropriately,  the partition structure of spatial $\Lambda$-coalescents of individuals sampled far enough apart converges to the partition structure of a non-spatial Kingman coalescent. From a biological point of view this means that in certain circumstances both the spatial structure as well as larger variances of the underlying offspring distribution  are harder to detect from the sample. However, supplemental simulations show that for moderately large $L$ the different structure is still evident. 

\end{abstract}

\begin{keyword}
spatial Cannings model, coalescent, $\Lambda$-coalescent, spatial coalescent, two dimensional torus, limit theorems
%% keywords here, in the form: keyword \sep keyword

%% MSC codes here, in the form: \MSC code \sep code
%% or \MSC[2008] code \sep code (2000 is the default)
\MSC[2010] 60J25 %Continuous-time Markov processes on general state spaces
\sep 60K35 %Interacting random processes; statistical mechanics type models; percolation theory62H11 Directional data; spatial statistics 
\sep 92D25 %Population dynamics (general) 
\sep 92D10 %Genetics
\end{keyword}

\end{frontmatter}

%%
%% Start line numbering here if you want
%%
% \linenumbers

%% main text

\section{Introduction}
\label{introduction}

\noindent
%We study genealogies for spatial models with large offspring variance (detail on large offspring variance)
The \emph{goal} of this article is to study the genealogies of a sample of individuals from a spatially structured population when the variance in the number of each individual's offspring is relatively large. Larger variances in the offspring distribution are thought to arise, for example,  due to particular reproduction mechanisms of various species   leading to the existence of few individuals with many offspring (Eldon and Wakeley, 2006)\nocite{EW06}, and also due to recurring selective sweeps (Durrett and Schweinsberg, 2004, 2005) \nocite{DS04, DS05}.

%Detail on the space: generalization of stepping stone model
The space is structured into discrete sites of a graph $G$ with a colony of a fixed number of individuals at each site in $G$ as well as migration between sites. As the underlying population models we introduce spatial Cannings models, which are extensions of the stepping stone model %Kimura 53, see also ZŠhle, Cox, Durrett
with general Cannings type offspring distributions. 

%Genealogies are modeled by coalescent processes, and approximated for large N and then large L
The genealogies are modeled by coalescent processes that code for the ancestral lines of the individuals in a sample from the present day population backwards in time. Coalescence -referring to a merger of ancestral lines- occurs when a common ancestor of various individuals is reached. We consider coalescent models that are appropriate if the population at each site of $G$ is large. Our special focus in the analysis  lies then on models where the number of sites $|G|$ in the graph is finite but also large, more precisely, we choose $G$ to be a large two dimensional torus.

%Special case: nonspatial Kingman 
The special case when one considers only one site (the non-spatial situation with $|G|=1$) leads to the classical \emph{Kingman coalescent} with only binary mergers provided that the variance of the offspring distribution stays bounded in some sense as the population size $N$ 
tends to infinity. This coalescent process has been well studied since its introduction by Kingman (1982a; 1982b)\nocite{jK82a, jK82b}, see for example Wakeley (2009) \nocite{jW09} for an overview. 

%...and nonspatial Lambda with Kingman as its special case
Genealogies for populations with a larger variance in the number of offspring have been studied by Mathematicians and Biologists in the more recent past. 
In this case, the genealogy is described by the coalescent with multiple mergers,  the so called \emph{$\Lambda$-coalescent}, which was independently introduced by Pitman (1999) \nocite{jP99} and Sagitov (1999)\nocite{sS99}. (More generally, one can even consider coalescents with simultaneous multiple mergers, the so called $\Xi$-coalescents, but we will focus here on the subclass of $\Lambda$-coalescents.) Here, $\Lambda$ is a finite measure on $[0,1].$ When there are currently $b$ distinct ancestral lines then any collection of $k$ ancestral lines coalesces and thus merges into one new ancestral line at rate 
\begin{equation}
\label{Edeflambk}
\lambda_{b,k}:=\int_{[0,1]} z^{k-2}(1-z)^{b-k} \Lambda(dz)
\end{equation}
with $2\leq k\leq b$, $k,b \in \IN.$ We formally extend this definition by setting $\lambda_{b,k}=0$ for $b=1$ or $b=0$, $k\in \IN.$
The  Kingman coalescent, that is appropriate for populations in which the offspring variance is not so large, corresponds to the case $\Lambda=\delta_0,$ the delta measure at $0.$  In this case $\lambda_{b,k}$ is only nontrivial if $k=2$ and we obtain $\lambda_{b,2}=1$ so that the total coalescence rate is $\binom{b}{2}.$ This means that only binary mergers of pairs of ancestral lines are possible and happen at rate $1$ per pair of ancestral lines.
 
%
%Reference to LAMBDA COALESCENTS DEFINITION in following Section
%
In order to extend this to a \emph{spatial or structured coalescent}  on a graph $G$ we imagine that ancestral lines  migrate independently from site $x$ to site $y$ in $G$ with rates $p(x,y)$ but coalesce independently at each site according to the rates given in (\ref{Edeflambk}). The formal definition of this \emph{spatial $\Lambda$-coalescent} can be found in Section \ref{section:spatialLambdadef} and in particular in Definition \ref{defn:spatialLambda}.  In the case of Kingman coalescence at each site the resulting process is often simply called the structured coalescent, which was derived rigorously from a forward population model by Herbots (1997) \nocite{hH97} but already studied earlier by Notohara (1990) \nocite{mN90} and since then by many others, see for example Donnelly and Kurtz (1999) or Greven et al. (2005)\nocite{DK99, GLW05}. The case when coalescence at each site takes place according to a $\Lambda$-coalescent, was included in a setting considered by Donnelly and Kurtz (1999)\nocite{DK99}, and was studied in more detail by Limic and Sturm (2006)\nocite{LS06}.

%
%Reference to Section \ref{section:SpatialCannings-SpatialLambda}  
%
After introducing the spatial Cannings model in Section \ref{section:SpatialCannings-SpatialLambda}  
we show that spatial $\Lambda$-coalescents arise when one considers the genealogy of individuals sampled in those spatial Cannings models, provided that the number of individuals at each site is large, see  Proposition \ref{thm:conv_to_spatialLambda} for a precise formulation of the result. This result is a rather straightforward generalisation of  the work of Herbots (1997) \nocite{hH97} on the derivation of the structured coalescent combined with the work of M\"ohle and Sagitov (2001) \nocite{MS01} on the derivation of $\Lambda$-coalescents from non-spatial Cannings models, which is included here for completeness.

%
%Reference to Section \ref{section:spatialLambdaontorus} and finally Section \ref{section:ssion}
%
In Section \ref{section:spatialLambdaontorus} we consider the spatial $\Lambda$-coalescent on  the two dimensional torus $T^L=[-L,L]^2 \cap \IZ^2$  with side length $2L+1.$ In our main result, Theorem \ref{thm:torusconv}, we show that as $L$ tends to infinity, and time is rescaled appropriately,  the partition structure of spatial $\Lambda$-coalescents of individuals sampled far enough apart converges to the partition structure of a non-spatial Kingman coalescent. This extends in particular the work of Limic and Sturm (2006)\nocite{LS06}, which dealt with the case $d \geq 3.$ The two dimensional case considered here  is biologically the most relevant. It also needs to be treated differently mathematically.

Since the information in a genetic sample only depends on the partition structure of the genealogy of the sampled individuals as well as on some (independent) mutation process along the ancestral lineages,  the distribution of quantities that can be read off a sample of genes, such as allele frequencies,  will be the same if the partition structures of the genealogies are the same. (Note however that the time rescaling that is used in Theorem \ref{thm:torusconv} depends explicitly on $L.$)

From a biological point of view, Theorem \ref{thm:torusconv}  thus implies  that spatial structure as well as larger variances of the underlying offspring distribution are harder to detect from a sparse sample on a large space. However, we also conduct simulations that show that for moderately large $L$ the different structure is still evident in the frequency spectrum of a sample. 
We postpone the detailed description of relations of our main limit result  to the existing literature as well as further discussion of the analytical as well as simulation results to 
Section \ref{section:discussion}.

\section{The spatial $\Lambda$-coalescent}
\label{section:spatialLambdadef}
% 
%PARTITION VALUED DESCRIPTION
% 
In this section, we rigorously define the spatial $\Lambda$-coalescent. We start with introducing some necessary notation.
First, in order to not only describe the number of ancestors of a sample of size $n$ at a given time, but the structure of the genealogy, it is convenient to consider the coalescent as partition valued. We label the sampled individuals from $1$ to $n$ and want to describe the corresponding $n$-coalescent, first in a non-spatial setting.

Here, we consider as the state space  ${\cal P}_n,$  the partitions of $[n]:=\{1,\dots, n\}.$ We call the elements of a partition $\pi \in  {\cal P}_n$ blocks and may represent 
$\pi$ uniquely by
\begin{equation}
\label{blockrep}
{\pi}:=(B_1, B_2, \dotsc),
\end{equation}
where $B_l \subset [n]$ with $\sum B_l = [n]$ are the blocks indexed by the order of their minimal element (with the convention that $\min \emptyset = \infty$).
Each block represents the common ancestor of the individuals contained in this block. Initially, the process is generally started with all singletons, so in the state $(\{1\},\dots,\{n\}, \emptyset, \dots ).$ At each coalescence event an appropriate collection of blocks  is merged into one new larger block.  

$\Lambda$-coalescents can more generally be defined on the state space ${\cal P},$ the partitions of all of $\IN$ with representations as in (\ref{blockrep}). Due to the exchangeability of all blocks an $n$-coalescent is obtained from this coalescent started with infinitely many individuals if one restricts the attention to the partition structure induced on a subset of $\IN$ of size $n.$

%
%PARTITION VALUED PROCESSES DEFINITION
%
For describing the spatial distribution of ancestral lineages of samples of size $n \in \IN$ we consider as the state space for the spatial $\Lambda$-coalescent labeled partitions ${\cal P}^{\ell}_n$ of  $[n].$ This is the set of partitions of $[n],$ for which each block
carries a label in $G$ specifying its location. 
To make the notation more precise, we will represent an element 
$\pi^{\ell} \in {\cal P}^{\ell}_n$ uniquely by 
\begin{equation}
\label{labeledblocks}
{\pi}^{\ell}:=((B_1,\zeta_1), (B_2,\zeta_2), \dotsc),
\end{equation}
where the sequence of blocks is ordered as before and
$\zeta_l \in G, l \in \IN$ are the labels of the ordered blocks. Set $\zeta_l=\partial \not\in G$ if
$B_{l}=\emptyset$.

If we are interested in the genealogy of samples of infinite size we consider labeled partitions ${\cal P}^{\ell}$ of $\IN.$ Elements of ${\cal P}^{\ell}$ are also uniquely described
by (\ref{labeledblocks}) with $B_l \subset \IN$, where $\sum B_l = \IN$ and $\zeta_l \in G$.
For ${\pi}^{\ell}$ in ${\cal P}^{\ell}_n$ or ${\cal P}^{\ell}$ of the form (\ref{labeledblocks}) we will write 
$\pi$ for the corresponding unlabeled partition $\pi:=(B_1, B_2, \dots)$ in 
${\cal P}_n$ or ${\cal P}$ respectively. We will denote by $\#\pi=\#\pi^{\ell}$ the number of (nonempty) blocks in 
the partition, which may be finite or infinite. Also write $\#\pi^{\ell}_x$ for the number of blocks with label $x \in G$.

For any element $\pi^{\ell} \in
{\cal P}^{\ell}_n$ or ${\cal P}^{\ell}$ with $n\geq m$ 
define $\pi^{\ell}|_m \in {\cal P}^{\ell}_m$ as the labeled partition
induced by $\pi^{\ell}$ on ${\cal P}^{\ell}_m$.
We equip ${\cal P}^{\ell}_n$ with the
metric
\begin{equation}
\label{E_nmetric}
d_n(\pi^{\ell,1},\pi^{\ell,2})= \sup_{m \in [n]} 2^{-m}1_{\{\pi^{\ell,1}|_m \neq \pi^{\ell,2}|_m\}},
\end{equation}
and likewise ${\cal P}^{\ell}$ with the analogous metric $d,$ where we replace $[n]$ with $\N.$
It follows that $({\cal P}^{\ell}_n,d_n)$ and $({\cal P}^{\ell},d)$ are \emph{both} 
compact Polish spaces (for finite $G$), and that 
$d(\pi^{\ell,1},\pi^{\ell,2}) = \sup_n d_n(\pi^{\ell,1}|_n,\pi^{\ell,2}|_n)$.
Analogously, we defined a metric on the unlabeled partition spaces ${\cal P}_n$ and ${\cal P}$ 
by omitting all the superscripts ${\ell}$ in the above.

%
%SPATIAL LAMBDA COALESCENT DEFINITION
%
We are now ready to  rigorously define the spatial $\Lambda$-coalescent. Let $D(\IR_+, E)$ be the space of right continuous functions from $\IR_+$ into a metric space $E$ that also have left limits. We equip the space $D(\IR_+, E)$ with the Skorohod topology, see for example Chapter 3 of Ethier and Kurtz (1986) \nocite{EK} for details.
\begin{defn}
\label{defn:spatialLambda}
The spatial $\Lambda$-coalescent $\Pi^{\ell}$ with parameter $\Lambda= (\Lambda_{x})_{x \in G}$
where $\Lambda_x$ are finite measures on $[0,1]$ is a $D(\IR_+, {\cal P}^{\ell})$-valued process
with the following dynamics:
\begin{itemize}
\item[(i)] Blocks with the same label $x$ coalesce to create a new block with that same label according to the (non-spatial) $\Lambda_x$-coalescent independently from blocks at other sites, meaning that coalescence happens with the rates given in (\ref{Edeflambk}) for $\Lambda=\Lambda_x.$
\item[(ii)] The label of each block performs an independent random walk  on $G$
with the migration rate from site $x$ to site $y$ given by  $p(x,y), x,y \in G$.
\end{itemize}
\end{defn}
It is shown in Limic and Sturm (2006) \nocite{LS06} Theorem 1 that the spatial $\Lambda$-coalescent is a well-defined strong Markov process for any initial condition in ${\cal P}^{\ell}$ in the case that $\Lambda_x$
are the same measures for all $x \in G$ and $|G|< \infty.$ However, the construction can easily be extended  to this more general setting and also to $|G|$ countably infinite
(some care must be taken here for starting configuration with infinitely many blocks, but we will not need this here).
As mentioned earlier, the special case for which $\Lambda_x$ is given by $\delta_0$ up to a constant for all $x \in G$ is called the structured coalescent (or spatial
Kingman coalescent).

\section{Spatial Cannings models and derivation of spatial $\Lambda$ coalescents}
\label{section:SpatialCannings-SpatialLambda}
In this section we want to present a class of simple spatial population model, which we will refer to as \emph{spatial Cannings models.} We will show that the spatial genealogy of a sample of individuals from various sites is described by the spatial $\Lambda$-coalescent of Definition \ref{defn:spatialLambda} in the large population limit. Here, the number $|G|$ of sites in the graph may be finite or countably infinite. The convergence result combines work by Herbots (1997)\nocite{hH97}, who considered convergence to the structured coalescent, with results on convergence of genealogies to $\Lambda$-coalescents in the non-spatial setting from M\"ohle and Sagitov (2001)\nocite{MS01}. 

%
%SPATIAL CANNINGS
%
In the spatial Cannings model  there are
$N_x:=r_x N \in \IN$ individuals at  site $x \in G$ at any time, where $r_x$ is a constant. 
These populations reproduce within their own site at discrete times $s \in \IN_0:=\{0,1,2,\dots\}$ and immediately after birth disperse their offspring to other sites.  
Both the reproduction as well as the dispersal or migration mechanism are chosen such that they leave the population size 
at any site fixed. The offspring law at site $x \in G$ is described by 
\begin{equation}
\label{offspring}
\nu^x=(\nu_1^x,\nu_2^x,\dots, \nu_{N_x}^x),
\end{equation}
where $\nu_l^x$ is the number of offspring of individual $l$ in the previous generation (all at site $x$). 
In order to keep the population sizes constant  over time we assume 
that $\sum_{l=1}^{N_x} \nu_l^x=N_x.$ 

%
%EXCHANGEABILITY
%
For the offspring distributions $\nu^x$ we will furthermore assume that they are exchangeable such  that for any permutation $\sigma$ of the $N_x$ indices we have
\begin{equation}
(\nu_1^x,\nu_2^x,\dots, \nu_{N_x}^x) \overset{{\cal D}}{=} (\nu_{\sigma(1)}^x,\nu_{\sigma(2)}^x,\dots, \nu_{\sigma(N_x)}^x).
\end{equation} 
Cannings (1974)\nocite{cC74} first considered non-spatial versions of this model with offspring distributions satisfying these properties.

Due to the exchangeability and the criticality of the reproduction we have $\IE(\nu_l^x) = 1.$ 
The probability that any two individuals at site $x$ have a common ancestor in the previous generation is given by  
\begin{equation}
\label{c_Ndef}
c^N_{x} := \sum_{l=1}^{N_x} \IE\left( \frac{\nu_l^x (\nu_l^x-1)}{N_x (N_x-1)} \right)
= \frac{\IE((\nu_1^x)_{2})}{N_x-1}= \frac{Var(\nu_1^x)}{N_x-1},
\end{equation}
where we are  denoting $(m)_k:=\frac{m!}{(m-k)!}.$ 
%
%DISPERSAL
%
After reproduction a fixed number of offspring $n_{xy}$ selected at random without replacement from site $x$ migrate to site $y$ for all $x,y \in G.$ Thus, for each 
$x \in G,$ we need that $\sum_{y \in G}  n_{xy}\leq N_x.$ In order to  keep the population sizes at all sites fixed even after the dispersal due to migration we also require \emph{balancing migration}, meaning that 
\begin{equation}
\label{balance}
\sum_{y \neq x} n_{xy} = \sum_{y \neq x} n_{yx}.
\end{equation}
Taking reproduction and migration together specifies the spatial Cannings model fully. We note that the special case considered by Herbots (1997)\nocite{hH97} corresponds to Wright-Fisher reproduction in all colonies (which are taken to be of the same size $N$) such that $\nu_x$ of (\ref{offspring}) is given by a symmetric multinomial distribution for all $x.$ 

In order to study the associated genealogies of a sample we still need the following notation.
After reproduction and migration, the proportion of individuals in site $x$ who were born in site $y$ is given by 
\begin{equation}
p_{xy} = \frac{n_{yx}}{N_x} = \frac{N_y q_{yx}}{N_x} = \frac{r_y}{r_x} q_{yx}, 
\end{equation}
where $q_{yx}$ is the proportion of individuals in site $y$ leaving for site $x.$
We also set
$p_{x} = \sum_{y \neq x } p_{xy},$
which is the proportion of individuals at site $x$ that migrated there.

%
%CANNINGS COALESCENT DEFINITION
%
Now let the \emph{spatial Cannings-coalescent} $\Pi^{\ell, N} = (\Pi^{\ell, N}_{s})_{s \in \IN_0}$ be derived from a spatial Cannings model such that this ${\cal P}^{\ell}_n$ valued process is obtain from sampling $n$ individuals at the present, whose location is represented by $\pi^{\ell} \in {\cal P}^{\ell}_n,$ and then following their genealogy into the past.

In order to obtain a spatial $\Lambda$-coalescent in continuous time from Definition \ref{defn:spatialLambda} in the large population limit as $N \rightarrow \infty$ we 
 will rescale time for this process as a function of $N$ and let $N \rightarrow \infty.$ Set
\begin{equation}
t^N = \frac{t}{c^N}, 
\end{equation} 
where $c^N$ is related to $c^N_{x}$ and specified later on. Naturally, if $c^N_{x}$ are independent of $x \in G$ we will set $c^N=c^N_{x}.$
This would in particular be the case in a spatially homogeneous situation in which all colonies are of the same size and display the same reproductive behavior, 
meaning that $r_x=1$ so that $N_x=N$ and $\nu^N=\nu^{x,N}$ for all sites $x \in G.$

Further, we will need an asymptotic moment condition as in M\"ohle and Sagitov (2001)\nocite{MS01} stating that
\begin{equation}
\label{limmoments}
\phi_{j}^{x}(k_1,\dots,k_j):=\lim_{N \rightarrow \infty} (N^{k_1 + \cdots + k_j-j}c^N_{x} )^{-1}
\IE( (\nu_1^{x,N})_{k_1} \cdots (\nu_j^{x,N})_{k_j}) 
\end{equation}
exist  for all $x \in G$, $j  \in \IN$ and $k_1, \dots, k_j \geq 2.$  
%
%CONVERGENCE TO SPATIAL LAMBDA COALESCENT
%
We are now ready to state a convergence in distribution result for the spatial Cannings-coalescent.
\begin{propn}
\label{thm:conv_to_spatialLambda}
Assume that there exists a sequence $\{c^N\}_{N \in \IN} \subset \IR_+$ with $\lim_{N \rightarrow \infty} c^N =0$ and constants $c_x$ such that
for any $x \in G$ we have  $\lim_{N \rightarrow \infty} \frac{c^N_{x}}{c^N}= c_x$ 
with $\sup_{x \in G} c_x< \infty.$  Also assume that 
 $\lim_{N \rightarrow \infty } \frac{p_{xy}^N}{c^N} = p(x,y)$ where $\sup_{x \in G} \sum_{y \neq x} p(x,y) < \infty.$ Assume further that
$ \phi_{1}^{x}(k)$ exist for all $x \in G$  with $\sup_{x \in G}  \phi_{1}^{x}(k) < \infty$ for all $k \geq 2$ as well as
 \begin{equation}
\label{lambdacond}
\phi_{2}^{x}(2,2)
=\lim_{N \rightarrow \infty}( N^2 c^N_x)^{-1}\IE( (\nu_1^{x,N})_{2} \cdot (\nu_2^{x,N})_{2}) =0
\end{equation}
 for all $x \in G.$ For $n \in \IN$ and $\pi^{\ell} \in
{\cal P}^{\ell}_n$ let $\Pi^{\ell,N}$ be the spatial Cannings-coalescent 
started in $\pi^{\ell}.$ Then we obtain weak convergence
\begin{equation}
(\Pi^{\ell,N}_{ [\frac{t}{c^N}]})_{t \in \IR_+}\Rightarrow (\Pi^{\ell}_t)_{t \in \IR_+} \quad \text{ in } D(\IR_+,{\cal P}^{\ell}_n)
\end{equation}  
as $N \rightarrow \infty,$ where $ \Pi^{\ell}$ is the spatial $\Lambda$-coalescent as in Definition \ref{defn:spatialLambda} also started in $\pi^{\ell}$ with $\Lambda_x$ characterized
by the moments given by 
\begin{equation}
\label{Lambda-phi}
\int_0^1 z^{k-2} \Lambda_x(dz) =c_x \phi_1^x(k), \quad k \geq 2.
\end{equation}
\end{propn}
%
%
%CONVERGENCE CONDITIONS
%
\smallskip
The condition (\ref{lambdacond})
 implies that $\phi_{j}^x\equiv0$ for $j\geq 2$ and all $x \in G$ by monotonicity in $j,$ see (18) of M\"ohle and Sagitov (2001)\nocite{MS01}.
%
%EARLIER CONVERGENCE RESULTS
%
This also implies immediately that we obtain convergence to the spatial Kingman coalescent if and only if
\begin{equation}
\label{Kingmancond}
\phi_1^{x}(3)=\lim_{N \rightarrow \infty} (N^2 c^N_x)^{-1} E( (\nu_1^{N,x})_3 )=0.
\end{equation}
As mentioned earlier, a special case of Proposition \ref{thm:conv_to_spatialLambda} was proved by Herbots (1997)\nocite{hH97}.   For the Wright-Fisher reproduction she considers we set $c^N=\frac{1}{N}.$  Then under the same assumptions Herbots shows convergence of \emph{the number of blocks}   on the space $D(\IR_+, \IN^{|G|}).$ Since (\ref{Kingmancond}) is satisfied, the limit is the block counting process of the structured coalescent.  Here, we consider convergence of spatial partition valued processes in a more general setting.

Proposition \ref{thm:conv_to_spatialLambda} in the non-spatial situation for which $|G|=1$ was proved by M\"ohle and Sagitov (2001)\nocite{MS01}. They considered convergence of partition valued processes for even more general offspring distributions and proved that the limiting process $\Pi$ is a coalescent with simultaneous multiple collisions, or $\Xi$-coalescent, if and only if (\ref{limmoments}) is satisfied. 
Proposition \ref{thm:conv_to_spatialLambda} can easily be extended to those settings as well but since we will, for ease of notation, only be dealing with spatial $\Lambda$-coalescents later on we restrict out attention to this subclass.

%
%SECTION: BEHAVIOR ON THE TWO DIMENSIONAL TORUS
%
\section{Behavior of the spatial $\Lambda$-coalescent on a large 2-dimensional torus}
\label{section:spatialLambdaontorus}

From now on let $G=T^L$ and in order to record the dependence on $L$ we will write $\Ps_n^{\ell,L}$ and $\Ps^{\ell,L}$ for the partitions with labels in $T^L.$
We will  for notational simplicity assume that $\Lambda_x=\Lambda$ is constant for all $x \in T^L$ and that $\Lambda([0,1]) > 0$. 
(It really suffices to assume that $\lambda_{2,2}^x$ is bounded above and below by positive constants uniformly in $x \in T^L.$)
In order to specify the migration on the torus we first consider transition rates $\tilde{p} : \IZ^2 \to [0,1]$ and define for all $L \in \IN$ transition rates $p^L$ of a random walk on $T^L$ by
\begin{equation}
p^L(x,y) = \sum_{z \in \{z' \in \Z^2\; | \; z'-y \textup{ mod } 2L+1 = 0\}} \tilde{p}(x,z).
\end{equation}
This means that ancestral lines that "migrate out" on one side of  $T^L$ will "migrate in" again on the other side. 
To simplify the arguments in the following we will assume that the rates $\tilde{p}$ are those of a simple symmetric random walk on $\IZ^2$ that migrates to each of the four nearest neighbour sites with equal rates $\frac{1}{4}.$ However, any spatially homogeneous, symmetric random walk satisfying a suitable moment condition could also be considered (see also the remark at the end of this section).

On $T^L$ we also define the following metric $r_L$ appropriate for the torus, 
	\begin{equation}
	r_L(x,y):= \inf \{\norm{x-z}_2 | z \in \Z^2, y - z \textup{ mod } 2L+1 = 0\},
	\end{equation}
	where $\norm{x}_2:=\sqrt{x_1^2 + x_2^2}$ for $x \in   \Z^2.$ 
Furthermore, we define  for $0 \leq a \leq b$ the set of all labeled partitions with pairwise distances  in $[a,b]$:
	\begin{eqnarray}
	[[a,b]] &\defeq& \{{\pi}^{\ell}=((B_1,\zeta_1), (B_2,\zeta_2), \dotsc)  \in \Ps^{\ell,L}_n |\\
	\nonumber
	&&\quad \quad  r_L\bigl(\zeta_i,\zeta_j\bigr) \in [a,b] \text{ for all } i\neq j \text{ with }  B_i,  B_j \neq \emptyset \}.
	\end{eqnarray}
In order to describe the asymptotic behavior of the spatial $\Lambda$-$n$-coalescent started with $n$ individuals we need some more notation:
	For $\pi \in \Ps_n$ let 
	\begin{equation}
	K^{\pi} = \bigl(K^{\pi}_t \bigr)_{t \in \R_+}
	\end{equation}
	 be the \emph{non-spatial} Kingman coalescent on $[n]$ started in $\pi.$
	We also define the sequence $(s_L)_{L \in \N},$ which  will be used to rescale time, by
	 \begin{equation}\label{s_Ldef}
	 s_L\defeq (2L+1)^2 \log(2L+1).
	 \end{equation}
Our main theorem then states the following.
\begin{thm}
	\label{thm:torusconv}
	Let $n \geq 2$ and let $(a_L)_{L \in \N}$ be a nonnegative sequence with 
	\[\limty[L] L^{-1} \sqrt{\log L} \;a_L = \infty, \quad 
	\quad \limty[L] L^{-1} a_L = 0\] 
	and $a_L \leq \sqrt{2}L$ for all $L \in \N$.
	Also, let  $\pi_0 \in \Ps_n$ and $\pi^{\ell,L} \in \Ps_n^{\ell,L}$ such that 
	\[\pi^{\ell,L} \in [[a_L, \sqrt{2} L]]\]
	and $\pi^L = \pi_0$ for all $L \in \IN$ large enough.
	For  $L \in \N$ let $\Pi^{\ell,L}$ be a  spatial $\Lambda$-$n$-coalescent started in  $\pi^{\ell,L}$.
	Then, we obtain weak convergence in the Skorohod space $\textup{D}(\R_+,\Ps_n)$ for $L \to \infty,$ more precisely
	\begin{equation}
	\label{SpatialLambdaconv}
	\bigl(\Pi^L_{s_L t}\bigr)_{t \in \R_+} \konvD \bigl(K^{\pi_0}_{\boldsymbol{\pi} t}\bigr)_{t \in \R_+}.
	\end{equation}
\end{thm}
Note that we will generally assume that $L$ is large enough so that $\pi^{\ell,L} \in \Ps_n^{\ell,L}$ corresponding to a given $\pi_0 \in \Ps_n$  can be found.
Observe also that in the limiting non-spatial Kingman coalescent there is an additional time change by the number $\boldsymbol{\pi}.$ The time change and so the entire limit process does  not depend on  the details of the coalescence mechanism of the spatial$\Lambda$-coalescent.

Finally, we remark that the result of Theorem \ref{thm:torusconv} is also expected to hold if the random walk $\tilde{p}$ on $\IZ^2$ is symmetric and spatially homogeneous, and satisfies a suitable moment condition. In this case, the time scaling by $\boldsymbol{\pi}$ in (\ref{SpatialLambdaconv}) would have to be replaced by $\sigma^2 \boldsymbol{\pi}$ with $\sigma^2$ the variance of distance traversed by the random walk in one unit of time.

%
%DISCUSSION
%
\section{Discussion}
\label{section:discussion}

Knowledge of the genealogies together with an independent mutation process, in our context generally a Poisson process along the ancestral lineages, allows to describe the distribution of quantities of interest in population genetics that can be read off the genetic variability in the sample, such as site or allele frequency spectra. 

In the non-spatial situation many results are available for the Kingman coalescent case modeling populations with low offspring variance, a prominent example being the Ewens sampling formula for the allele frequency spectrum, see for example Wakeley (2009) \nocite{jW09}  for an overview.
For non-spatial  $\Lambda$-coalescents modeling populations with larger offspring variances the analysis is more complicated and the available results are not as complete or explicit. However, see for example M\"ohle (2006a; 2006b), Berestycki et al. (2007), Birkner et al. (2011), and Berestycki et al. (unpublished manuscript) \nocite{mM06a,mM06b,BBS07, BBS11, BBL12} for a variety of recent results that show that differences in the underlying reproduction can lead to qualitatively different sampling distributions. 

Likewise, the analysis of genetic variability in the sample becomes much more difficult due to spatial structure. Generally,  the influence of the spatial structure will again lead to qualitatively different sampling distributions. However, this depends on the underlying space and in particular on the distances of the sampled individuals.

Our main results state that  the genealogy of individuals sampled far enough apart on a large two dimensional torus $T^L$  can be approximated by the genealogy of a \emph{non-spatial} Kingman coalescent even when the offspring variances in the underlying population models are larger. On one hand, this means that the results available for the non-spatial Kingman coalescent can be used to approximate the genealogy and the resulting sampling distributions. However, from the point of view of a population geneticist it is a negative result: If one samples individuals relatively far apart then the influence of a large variance in the offspring distribution as well as of spatial structure are harder to detect in the sample.

In order to derive this result we consider spatial $\Lambda$-coalescents on $T^L,$ which model the genealogies of individuals sampled from large populations living and migrating on $T^L$ whose offspring distributions have  potentially larger variances. In Section \ref{section:SpatialCannings-SpatialLambda} we introduced spatial Cannings models as a large class of models that fit into this framework, a statement that is made precise by  Proposition \ref{thm:conv_to_spatialLambda}. 

The main result, Theorem \ref{thm:torusconv}, then states that the (unlabeled) partition structure of the suitably time changed spatial $\Lambda$-$n$-coalescents on $T^L$ converges to that of a non-spatial Kingman coalescent as $L \rightarrow \infty$ provided that the  individuals are sampled far enough apart. 
This kind of behavior arises since in the chosen scaling and due to the sparse sample  it is unlikely that more than two ancestral lines, represented by blocks, ever meet at the same site. Thus, only binary mergers may take place. On the other hand, a meeting of two  ancestral lines is followed up by many more meetings of these two so that they eventually coalesce (regardless of the rate of coalescence) before encountering any other  ancestral lines. The sequence of meetings of two ancestral lines and thus also their coalescence happens instantaneously as $L \rightarrow \infty,$ which implies that for large $L$ the spatial  $\Lambda$-$n$-coalescent behaves like a coalescing random walk (with instantaneous coalescence of lines that have met).
In contrast, the time between encounters and eventual coalescence of pairs of lines is long enough so that all ancestral lines have in the meantime become well mixed on the torus. Hence, any two of them are equally likely to participate in the next meeting, which leads in the limit to the exchangeability property of the non-spatial Kingman coalescent.

%
%Relations to other work
%

%
%Relations to Limic, Sturm und Greven, Limic, Winter
%
The result of Theorem \ref{thm:torusconv} is analogous to one obtained by  Limic and Sturm (2006) \nocite{LS06} for the $d$ dimensional torus with  $d\geq 3.$ However, the scaling for $d=2$ is different and more subtle. This is due to the fact that the random walk performed by the ancestral lines is recurrent in $d=2$ while it is transient for $d \geq 3.$ The recurrence in two dimensions also leads to the many encounters and in the limit $L \rightarrow \infty$ to instantaneous coalescence of a pair of lines that have met at the same site. As a consequence the time change of the limiting Kingman coalescent is independent of the $\Lambda$-measure of the underlying coalescent mechanism in $d=2$ unlike in $d=3.$ The prior work by Limic and Sturm (2006) \nocite{LS06} generalised results by Greven et al. (2005) \nocite{GLW05} who considered spatial Kingman coalescents in $d \geq 3$ (for some results in $d=2$ in the analogous setting but with a different focus see also Greven et al. (2012)\nocite{GLW12}).

 %
 %Relations to ZŠhle, Cox, Durrett
 %
 The observation that the influence of spatial structure on the genealogy of a sample is not readily detectable in certain situations (except possibly through a space dependent time change) 
 has been made before.  Related results have in particular been obtained in the articles by Cox and Durrett (2002) \nocite{CD02} and Z\"ahle et al. (2005)\nocite{ZCD05}. They study the classical stepping stone model on the torus $T^L.$ This  is a spatial Moran model  and thus a special case of the spatial Cannings model introduced in Section \ref{section:SpatialCannings-SpatialLambda} corresponding to setting $N_x=N$ and choosing $\nu_x^N=\nu^N$ in (\ref{offspring}) as a permutation of $(2,0,1, \dots,1).$ In Cox and Durrett (2002) \nocite{CD02} and Z\"ahle et al. (2005)\nocite{ZCD05} meeting times and coalescence times of ancestral lineages are analysed as $L$ tends to infinity while the populations size $N$ and migration rate depend on $L$ in an appropriate way. In this setting they also prove a result analogous to Theorem \ref{thm:torusconv}, stating that the genealogy of individuals sampled far enough apart can be approximated by a non-spatial Kingman coalescent.

%
%Relations to Barton, Etheridge and Veber
%
A related model in continuous space has recently been introduced by Barton et al. (2010)\nocite{BEV10}. Here, reproduction events are determined by Poisson point processes in space and involve individuals in a certain neighbourhood. Rare extinction-recolonisation events lead to a genealogy that is described by a spatial coalescent with a $\Lambda$-coalescent mechanism affecting ancestral lines in the neighbourhood chosen by the Poisson point process. 

These  $\Lambda$-coalescents in continuous space differ from the ones considered here in discrete space. Nevertheless, an analogous result to Theorem \ref{thm:torusconv} is obtained when individuals sampled far enough apart on a two dimensional continuous  torus are considered. As its side length $L$ tends to infinity, the limiting genealogy is described by a non-spatial Kingman coalescent, another type of a spatial $\Lambda$-coalescent or coalescing Brownian motions with non-local coalescence, depending on the scaling of the neighbourhood sizes that affect  multiple merger events.

In this work, the behavior of the spatial $\Lambda$-coalescent when individuals are not sampled far apart has not been considered analytically. It is clear, however, that due to the recurrence of the random walk in $d=2$ coalescence of all ancestral lines would happen instantaneously on the time scale considered in Theorem \ref{thm:torusconv} in the limit as $L \rightarrow \infty.$ 
This is in contrast to results in $d\geq 3$ as in Limic and Sturm (2006)\nocite{LS06},  where there is a nontrivial 
distribution of  ancestral lines (far apart from each other) for all large $L.$ 

For ancestral lines sampled close to each other the behavior of the ancestral lines is often described by distinguishing two phases, the first phase termed the \emph{scattering phase} that (likely) involves initial rapid coalescence and  lasts until ancestral lines are far apart (and well mixed) in space, and the second phase of subsequent slow coalescence termed the \emph{collecting phase}, see Wakeley (2001)\nocite{jW01}. 
Given this terminology, our analytical results are restricted to the collecting phase. For some analytical results on the genealogy for the scattering and collecting phase under various model assumptions, see Z\"ahle et al. (2005)\nocite{ZCD05}, Etheridge and V\'eber (in press 2012)\nocite{EV11}, and Greven et al. (unpublished manuscript)\nocite{GLW12}.

In order  to further examine the asymptotic results for the spatial $\Lambda$-coalescent we performed simulations of the spatial $\Lambda$-coalescent as well as the (non-spatial) Kingman coalescent on a moderately large two dimensional torus. We then compared the mean allele frequency spectrum of both processes which we juxtaposed additionally with its expectation in the Kingman coalescent calculated using Ewens sampling formula.
We also compared the total tree length of both processes via a q-q-plot.

The allele frequency spectrum is computed under the infinite alleles model. In order to test how well the approximation performs for moderately large $L$ mutations, that are assumed to always generate novel alleles, are placed on the genealogical tree of the processes according to a Poisson point process with the rate $\boldsymbol{\pi} \cdot ((2L+1)^2 \log (2L+1))^{-1}$. 
We start with $n$ singletons at time 0. Whenever a block is hit by a mutation we count the number of individuals in the block. Thus, if there are $k$ individuals in a block that is hit by a mutation a $k$-tuple of individuals carrying this mutation is  generated. We then remove the individuals from the system. When all individuals are removed we count the number of $k$-tuples, this number shall be called $a_k.$ The vector $(a_1 , \dotsc, a_n)$ is the desired frequency spectrum. We perform $m$ independent simulations and generate $a_i^j,$ which represent the $i$-th component of the allele frequency spectrum of the $j$-th simulation. The mean frequency spectrum is now given by $m^{-1}\sum_{j=1}^m (a_1^j,\dotsc, a_n^j)$.

\begin{figure}
\begin{minipage}[c]{0.5\textwidth}
\includegraphics[scale=1]{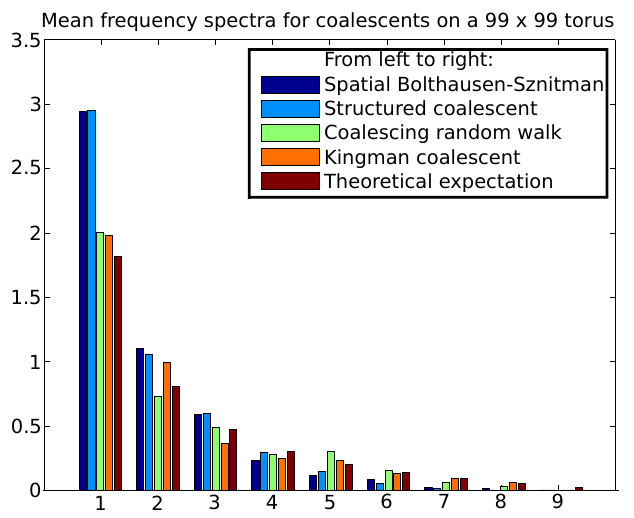} 
	\caption{Mean frequency spectrum for $L' = 99$}
	\label{fig1}
\end{minipage}
\begin{minipage}[c]{0.5\textwidth}
	\includegraphics[scale=1]{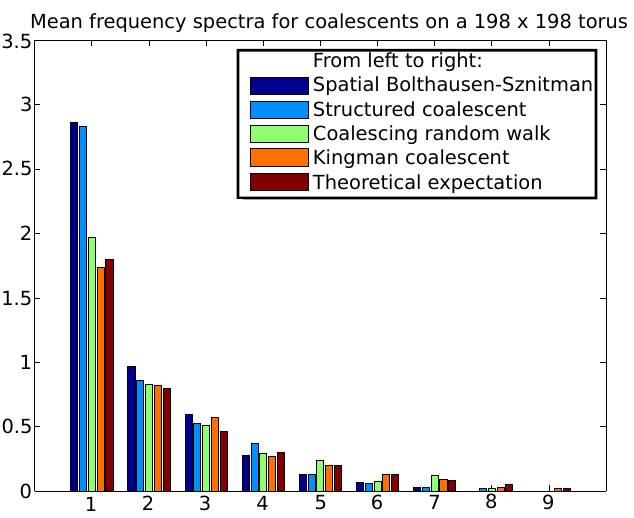} 
	\caption{Mean frequency spectrum for $L' = 198$}
	\label{fig2}
\end{minipage}
\end{figure}

First, we start with $n=9$ individuals sampled far apart on the torus at time $0.$ We arrange them on the torus in a 3 x 3 square such that the distance of next neighbours is a third of the side length $L'=2L+1$.

Regarding the various coalescence behaviors we consider the (instantaneously) coalescing random walk, the spatial Bolthausen-Sznitman coalescent ($\Lambda$ uniform on $[0,1]$) and the structured coalescent (the spatial Kingman coalescent with $\Lambda = \delta_0$). 
For the side lengths $L' = 99$ and $L' = 198$ we perform $m = 100$ independent simulations.

\begin{figure}
\begin{minipage}[c]{0.5\textwidth}
\includegraphics[scale=1]{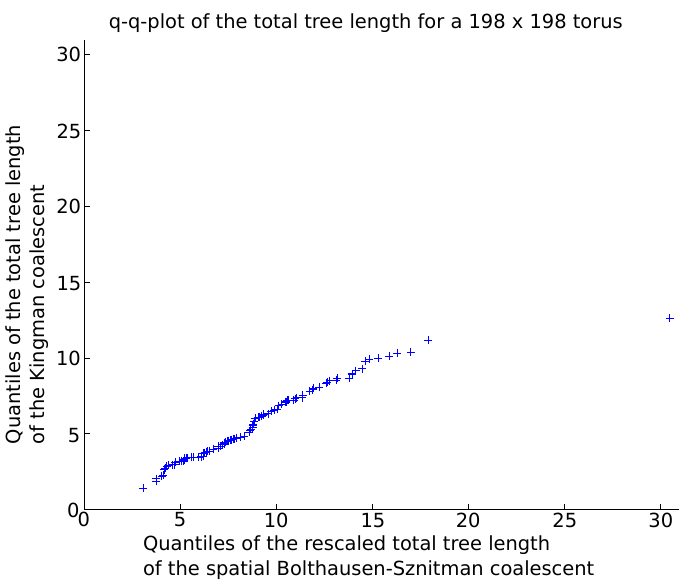} 
	\caption{q-q-plot of the rescaled total tree length of the spatial Bolthausen-Sznitman coalescent versus the total tree length of the Kingman coalescent for $L' = 198$}
	\label{fig3}
\end{minipage}
\begin{minipage}[c]{0.5\textwidth}
\includegraphics[scale=1]{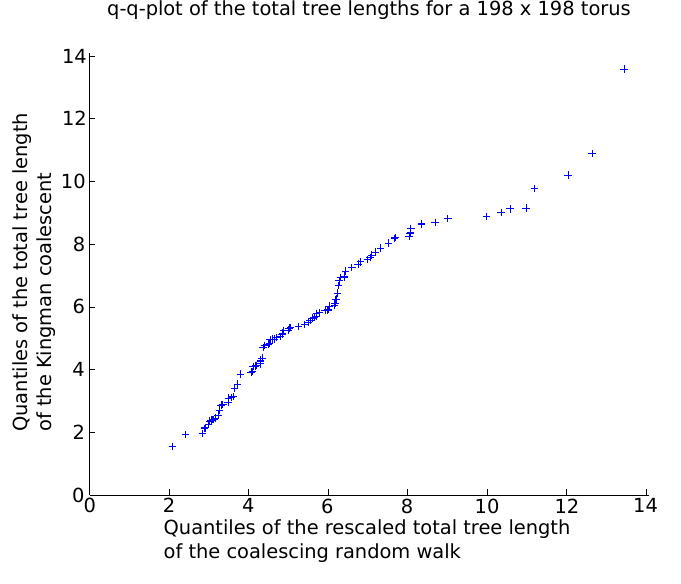} 
	\caption{q-q-plot of the rescaled total tree length of the coalescing random walk versus the total tree length of the Kingman coalescent for $L' = 198$}
	\label{fig4}
\end{minipage}
\end{figure}

Figures \ref{fig1} and \ref{fig2} show that the results for the coalescing random walk are much closer to the limiting non-spatial Kingman coalescent than those for the spatial $\Lambda$-coalescents.  The spatial $\Lambda$-coalescents produce more singletons in the allele frequency spectrum in comparison with the Kingman coalescent. Moreover, they also produce a larger total tree length (compare figures \ref{fig3} and \ref{fig4}). 

Comparing figures \ref{fig1} and \ref{fig2} we note in addition that there is not much of an improvement between the $L' = 99$ and $L' = 198$ case, which is a sign for a slow convergence rate in Theorem \ref{thm:torusconv}.
This can be explained by recalling the proof of the result one more time. We show the convergence result  first for the coalescing random walk.  We then show that for large $L$ the spatial $\Lambda$-coalescent behaves like a coalescing random walk since any pair of ancestral lines that meet in the spatial $\Lambda$-coalescent will with high probability on a large torus meet again and again (before meeting other ancestral lines) until the pair eventually coalesces. This additional time until coalescence vanishes in the limit. 
Nonetheless,  for moderately large $L$ it is not surprising that we see longer coalescence times for the spatial $\Lambda$-coalescents  than for the coalescing random walks and therefore also more singletons in the allele frequency spectrum since singletons have more time to get hit by a mutation before coalescing.  Note that the behavior of the Bolthausen-Sznitman coalescent and the structured coalescent are quite similar since it is mostly $\lambda_{2,2}$ that dictates how much longer it takes the spatial $\Lambda$-coalescent for coalescence, and we have $\lambda_{2,2}=1$ for both of these choices. 

The slow convergence rate may be explained by looking at the expectation of the time that it takes for two blocks to leave the same site and then meet again. This expected time  is of order $L^2$ and after rescaling it is of order $(\log L)^{-1}$. Therefore, these expectations converge only very slowly to zero.

\begin{figure}
\begin{minipage}[c]{0.5\textwidth}
\includegraphics[scale=1]{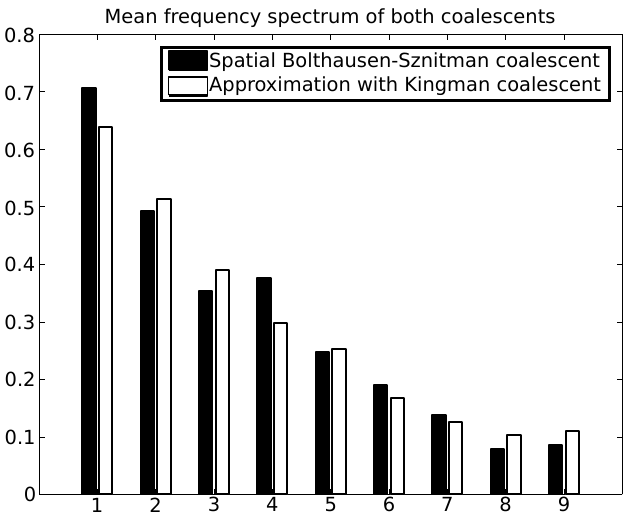} 
	\caption{Mean frequency spectrum of the Bolthausen-Sznitman coalescent and the Approximation with the Kingman coalescent with ancestral lines started close to each other for $L' = 99$}
	\label{fig5}
\end{minipage}
\begin{minipage}[c]{0.5\textwidth}
\includegraphics[scale=1]{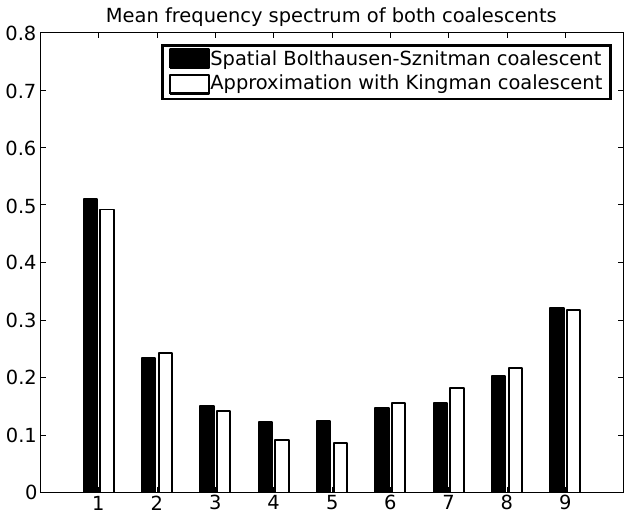} 
	\caption{Mean frequency spectrum of the Bolthausen-Sznitman coalescent and the Approximation with the Kingman coalescent with ancestral lines started at the same position for $L' = 99$}
	\label{fig6}
\end{minipage}
\end{figure}

Next, we consider individuals sampled close to each other on the torus and their behavior under the spatial Bolthausen-Sznitman coalescent.
Once the ancestral lines are mutually far apart again, we also use the approximation by a non-spatial Kingman coalescent. Therefore we get two results for the allele frequency spectrum, one given by the simulation of the spatial process and the other given by simulating until ancestral lines are far apart and then approximating with the Kingman coalescent.

Again we have $n=9$ individuals to start with, the side length of the torus is $L' = 99$ and we repeat the simulation $m=500$ times. We start the approximation when the ancestral lines surpass a mutual distance of $8.33$. We either start the particles in a 3 x 3 scheme where next neighbours have distance 1 or we start them all in the same site.

For both cases we compute the mean frequency spectrum (see figures \ref{fig5} and \ref{fig6}). 
In comparison with the sparse situation (figure \ref{fig1}) we observe that the mean of $a_1$ and $a_2$ decreases and the mean of $a_6, a_7, a_8$ and $a_9$ increases.  This effect is expected and is even stronger for individuals started at the same site.

We observe again that the approximation with the Kingman coalescent gives less weight to singletons but since the number of ancestral lines left in the system when it has become sparse will be very small (possibly 1) the approximation is sometimes not even used so that the results appear to be better than in the case where we already start far apart. Here, the mean number of ancestral lines at the time of approximation was $2.6$ for the close starting configuration and $4.0$ for all ancestral lines started in the same site.

Finally, we note that if ancestral lines end up far apart the simulation of the spatial $\Lambda$-coalescent will take significantly more time than the simulation of the approximation with the non-spatial Kingman coalescent, and thus could be of practical use. Here, the simulation with the approximation was more than 100 times faster than the simulation without the approximation.

%
%Conclusion
%
\smallskip

In conclusion, we recall that our theoretical convergence result states that information about spatial structure and the specific coalescence mechanism (due to possibly larger variances in the offspring distribution) is lost from a sparse sample as $L \rightarrow \infty,$ except for a time change that depends on $L.$ However, we have seen from simulations that even if this time change is not taken into account the influence of the spatial structure is still quite evident for moderately large $L.$ In contrast,  the details of the coalescence mechanism -apart from the delay that it takes a pair of lineages to coalesce while at the same site- do not have a significant impact on the sampling distribution for moderately large $L$ in the case of a sparse sample. 

In future theoretical work, it would be interesting to study the rate of convergence, and also whether it can be improved by a finer scaling that takes the coalescence delay into account. Also, theoretical investigation of the scattering phase for non sparse samples as well as the consideration of convergence as $N$ and $L$ tend to infinity jointly would be of interest.

\section{Proofs}
%
%PROOFS
%
In this section we prove Theorem \ref{thm:conv_to_spatialLambda} and Theorem \ref{thm:torusconv}.
\noindent
\subsection{Proof of {T}heorem \ref{thm:conv_to_spatialLambda}}
For the proof we fix the following notation. For a vector $a=(a_x)_{x \in G}$ we set $||a||=\sup_{x \in G} |a_x|.$
For $(a^N)_{N \in \IN}$ and  $(c^N)_{N \in \IN}$ sequences of vectors and nonnegative numbers respectively  we say that $a^{N}=O(c^{N})$ if $\sup_{N\in \IN} \frac{||a^{N}||}{  c^N} <\infty,$ and 
 $a^{N}=o(c^{N})$ if $\lim_{N\rightarrow \infty} \frac{||a^{N}||}{  c^N} =0.$ If $A=(a_{ij})_{i,j \in I}$ is a matrix for a countable $I$ then we define the matrix norm by 
\begin{equation}
||A||= \sup_{i \in I} \sum_{j \in I} |a_{ij}|.
\end{equation}
In this proof we only consider labeled partitions and thus omit the superscript ${\ell}$ in the notation whenever possible.
So for $n \leq N$ let $\pi, \tilde{\pi} \in {\cal P}^{\ell}_n$ and let $p^{N}_{\pi \tilde{\pi}}= \IP( \Pi^{\ell,N}_{t+1}=\tilde{\pi}|\Pi^{\ell,N}_t=\pi)$ be the transition probabilities of $ \Pi^{\ell,N}.$ We set $P_N= (p^{N}_{\pi \tilde{\pi}})_{\pi, \tilde{\pi} \in   {\cal P}^{\ell}_n}.$ Also let 
$q_{\pi \tilde{\pi}}$ be the transition rate from $\pi$ to $\tilde{\pi}$ for $\Pi^{\ell}$ if $\pi \neq \tilde{\pi}$ and set $q_{\pi \pi} = -\sum_{\tilde{\pi} \neq \pi} q_{\pi \tilde{\pi}}$ as well as $Q= (q_{\pi \tilde{\pi}})_{\pi, \tilde{\pi} \in   {\cal P}^{\ell}_n}.$
For the proof of the convergence of finite dimensional distributions  it suffices to show that
\begin{equation}
\label{asymp}
P_N = I + c^N Q + o(c^{N}).
\end{equation}
The sufficiency follows since  $||P_N ||= 1$ and also $|| I + c^N Q||=1$ for $N$ large due to the fact that the assumptions of the proposition imply 
that $q_{\pi \tilde{\pi}}$ are uniformly bounded over all $\pi, \tilde{\pi}  \in  {\cal P}^{\ell}_n.$
From this and (\ref{asymp}) we then have
 \begin{equation}
\label{Pconv}
|| P_N^{[\frac{t}{c^N}]} - (I + c^N Q)^{[\frac{t}{c^N}]}|| \leq  \frac{t}{c^N} || P_N - (I + c^N Q)||
= t || \frac{P_N - I}{c^N} - Q || \rightarrow 0
\end{equation}
as $N \rightarrow \infty$ which implies that 
\begin{equation}
\lim_{N \rightarrow \infty} P_N^{[\frac{t}{c^N}]} = \lim_{N \rightarrow \infty}  (I + c^N Q )^{[\frac{t}{c^N}]}
=e^{tQ}.
\end{equation}
This in turn implies convergence of the finite dimensional distributions.

Thus, we will first focus on showing relation (\ref{asymp}). For notational simplicity we will in the following omit a sub- or superscript $N$  whenever this does not affect clarity such as for the quantities $p_{xy}^N$ and $p_x^N$. 
We can write $P_N = P_N^{(m)}  \cdot P_N^{(r)}$ where $P_N^{(m)} =(p^{(N,m)}_{\pi \tilde{\pi}})_{\pi, \tilde{\pi} \in   {\cal P}^{\ell}_n}$ and  $P_N^{(r)}=(p^{(N,r)}_{\pi \tilde{\pi}})_{\pi, \tilde{\pi} \in   {\cal P}^{\ell}_n}$ are the transition probability matrices of the coalescent due to the migration step and reproduction step respectively. We first consider the migration step. The arguments will be similar to those in the proof of Theorem 2.1 in Herbots (1997)\nocite{hH97}.  If $\tilde{\pi} \in A^{\pi}_{xy}$  the set of all partitions that result from $\pi$ by changing one label from $x$ to $y$ with $x \neq y,$ then we have that
\begin{equation}
p^{(N,m)}_{\pi \tilde{\pi}} = \frac{\binom{N_x - \#\pi_x}{p_{xy} N_x -1}}{\binom{N_x}{p_{xy} N_x} } \cdot \frac{\binom{N_x - p_{xy} N_x - \#\pi_x +1}{p_x N_x - p_{xy} N_x}}{\binom{N_x - p_{xy} N_x }{p_x N_x - p_{xy} N_x}} \prod_{z \neq x} 
\frac{ \binom{N_z - \# \pi_z}{p_z N_z}}{\binom{N_k}{p_z N_K}}.
\end{equation}
Here, the first factor is the probability of choosing one particular block (but no other blocks) as migrants from site $y$ to site $x$ forward in time. The second factor represents the probability that
all other migrants from other sites to $x$ are drawn from outside the sample. Lastly, the product
represents the probability that no migrants at sites other than $x$ are in the sample. We can simplify
this expression to 
\begin{equation}
\label{p_pi_tildepi}
p^{(N,m)}_{\pi \tilde{\pi}} = p_{xy} \frac{N_x}{N_x - p_x N_x - \#\pi_x + 1}
\prod_{z \in G} \prod_{a=0}^{\#\pi_z-1} \frac{N_z - p_z N_z - a}{N_z-a}.
\end{equation}
We also have that 
\begin{equation}
\label{p_rest}
\sum_{\tilde{\pi} \notin \bigcup_{x\neq y} A^{\pi}_{xy}\cup\{\pi\}}  p^{(N,m)}_{\pi \tilde{\pi}}\leq R^{(N,m)}(\pi),
\end{equation}
where $R^{(N,m)}(\pi)$ is the probability that at least two of the migrants are drawn from the sample. As in 
Herbots (1997)\nocite{hH97} we can bound 
\begin{equation}
\label{p_restbound}
R^{(N,m)}(\pi) \leq \left(\sum_{z \in G} \#\pi_{z} p_z\right)^2\leq n^2 \sup_{z \in G} p_z^2=O((c^N)^2)
\end{equation}
by assumption. Taking (\ref{p_pi_tildepi}) to (\ref{p_restbound}) together now implies that
\begin{equation}
\label{asymp_m}
P_N^{(m)} = I + c^N Q^{(m)}_N + R_N^{(m)},
\end{equation}
where the entries of $Q^{(m)}_N=(q^{(N, m)}_{\pi \tilde{\pi}})_{\pi, \tilde{\pi} \in   {\cal P}^{\ell}_n}$
are given by $q^{(N,m)}_{\pi \tilde{\pi}}=\frac{1}{c^N} p^{(N,m)}_{\pi \tilde{\pi}}$ if $\tilde{\pi} \in A^{\pi}_{xy}$ for some $x \neq y ,$ by $q^{(N,m)}_{\pi \pi}=-\sum_{x \neq y} \sum_{\tilde{\pi} \in A^{\pi}_{xy}} \frac{1}{c^N} p^{(N,m)}_{\pi \tilde{\pi}}$ and are zero otherwise, and  $R_N^{(m)}$ is the matrix containing the rest terms. We have due to (\ref{p_rest}) and (\ref{p_restbound}) that 
 $||R_N^{(m)}||= O((c^{N})^2).$ Before turning to transition probabilities due to reproduction let us observe that 
\begin{equation}
\label{qmconv}
\lim_{N \rightarrow \infty} q^{(N,m)}_{\pi \tilde{\pi} }=q^{(m)}_{\pi \tilde{\pi} },
\end{equation}
where  $q^{(m)}_{\pi \tilde{\pi} }=p(x,y)$ if $ \tilde{\pi} \in  A^{\pi}_{xy}$ and is zero for all other $\tilde{\pi} \neq \pi.$ 
This convergence follows directly from (\ref{p_pi_tildepi}) since the term multiplying $p_{xy}$ in (\ref{p_pi_tildepi}) is a finite product whose individual factors converge to $1$ due to the fact that $N_z \rightarrow \infty$ and $p_z = O(c^N)\rightarrow 0$ by assumption for each $z \in G.$ From (\ref{qmconv}) it then follows that $q^{(m)}_{\pi \pi }= -\sum_{\tilde{\pi} \neq \pi} q^{(m)}_{\pi \tilde{\pi} }$ so that $Q^{(m)}_N=(q^{( m)}_{\pi \tilde{\pi}})_{\pi, \tilde{\pi} \in   {\cal P}^{\ell}_n}$  is the matrix of the transition rates for the migration in the limit.

\bigskip
Recall now that $P_N^{(r)} =(p^{(N,r)}_{\pi \tilde{\pi}})_{\pi, \tilde{\pi} \in   {\cal P}^{\ell}_n}$ describe the transition probabilities due to coalescence in the sample as a consequence of  the Cannings reproduction forward in time. If $\tilde{\pi}$ arises from $\pi$ by merging $2\leq k=\#\pi- \#\tilde{\pi}+1 $ blocks with the same label $x$ then
(28) in M\"ohle and Sagitov (2001)\nocite{MS01} states that 
\begin{equation}
\label{q^r_p^r}
q^{(r)}_{\pi \tilde{\pi}} := \lim_{N \rightarrow \infty} \frac{  1}{c^{N}} p^{(N,r)}_{\pi \tilde{\pi} }
=\lim_{N \rightarrow \infty}  \frac{c^N_{x}}{c^{N}}   \cdot \frac{  1}{c^N_{x}}p^{(N,r)}_{\pi \tilde{\pi} }= c_x \lambda_{\#\pi,k}^x.
\end{equation}
with $\lambda^x_{b,k}$ defined as in (\ref{Edeflambk}) with $\Lambda_x$ instead of $\Lambda.$
For other $\tilde{\pi}\neq \pi$ we have that $q^{(r)}_{\pi \tilde{\pi}}: =\lim_{N \rightarrow \infty} \frac{1}{c^N}p^{(N,r)}_{ \pi \tilde{\pi}}=0.$ 
Setting $q^{(r)}_{\pi \pi} =-\sum_{\tilde{\pi} \neq \pi} q^{(r)}_{\pi \tilde{\pi} }$ we see that
relation (\ref{q^r_p^r}) implies the analogous statement to (\ref{asymp_m}) for coalescence, 
\begin{equation}
\label{asymp_r}
P_N^{(r)} = I + c^N Q^{(r)} + R_N^{(r)},
\end{equation}
where $R_N^{(r)}=(r^{(N,r)}_{\pi \tilde{\pi}})_{\pi, \tilde{\pi}\in   {\cal P}^{\ell}_n}$ is the matrix of the rest terms.
Since  there are only finitely many (at most $2^n$) non-zero entries in each row we obtain $\sum_{\tilde{\pi} \in  {\cal P}^{\ell}_n} |r^{(N,r)}_{\pi \tilde{\pi}}|=o(c^N).$ 
\detail{If all sites have the same population dynamics then there are only finitely many different values for these sums since any $\pi \in {\cal P}^{\ell}_n$ with the same corresponding set $\{\#\pi_x, x \in G\}$ has the same set of  transition rates. 
Otherwise, we need some kind of uniformity condition in $x$ for (\ref{q^r_p^r}) at this point.}
Taken together this leads to  $||R_N^{(r)}||= o(c^N).$ 
 This implies that 
\begin{equation} 
P_N = I + c^N (Q_N + R_N),
\end{equation}
where
\begin{eqnarray}
Q_N&=& Q^{(m)}_N+Q^{(r)},\\
R_N&=&\frac{1}{c^N} ( R_N^{(m)} +R_N^{(r)} +R_N^{(m)}R_N^{(r)}
+c^N Q_N^{(m)}R_N^{(r)}+c^N R_N^{(m)} Q^{(r)}+(c^N)^2 Q_N^{(m)}Q^{(r)}).
\end{eqnarray}
with $|| R_N || \rightarrow 0$ as $N \rightarrow \infty$ since  $||R_N^{(m)}||= O((c^N)^2)$ and  $||R_N^{(r)}||= o(c^N).$ 
This and the convergence of $Q^{(m)}_N$ to $Q^{(m)}$ from (\ref{qmconv}) now completes the proof of 
(\ref{asymp}) and so of convergence of finite dimensional distributions.

In order to complete the proof of Theorem \ref{thm:conv_to_spatialLambda}  which states convergence 
in the Skorohod space $D(\IR_+,{\cal P}^{\ell}_n),$ it remains to show relative compactness in $D(\IR_+,{\cal P}^{\ell}_n).$ According to Corollary 3.7.4 in Ethier and Kurtz (1986)\nocite{EK} we need to show the following two conditions:
\begin{itemize}
\item[(i)] For every $\epsilon >0$ and $t\geq 0$ there exists a compact set $\Gamma_{\epsilon, t} \subseteq {\cal P}^{\ell}_n$ such that
$$ \liminf_{N \rightarrow \infty} \IP(\Pi^{\ell,N}_{[\frac{t}{c^N}]} \in  \Gamma_{\epsilon, t}) \geq 1-\epsilon.$$
\item[(ii)] For every $\epsilon >0$ and $t\geq 0$ there exists a $\delta>0$ such that 
$$\limsup_{ N \rightarrow \infty} \IP( w(\Pi^{\ell,N}_{[\frac{\cdot}{c^N}]}, \delta, T) \geq \epsilon) \leq \epsilon$$

where 
\begin{equation}\label{modcont}
w(\Pi^{\ell,N}_{[\frac{\cdot}{c^N}]}, \delta, T)= \inf_{\{t_i\}}\max_{i} \sup_{s,t \in [t_{i-1},t_i)} 
d_n(\Pi^{\ell,N}_{[\frac{s}{c^N}]},\Pi^{\ell,N}_{[\frac{t}{c^N}]}),
\end{equation}
and $\{t_i\}$ ranges over all partitions of the form $0=t_0<t_1< \cdots< t_{k-1}< T \leq t_{k}$ with 
$\min_{1 \leq i \leq k} (t_i - t_{i-1}) > \delta$ and $k \geq 1.$
\end{itemize}

\noindent
Since $\Ps_n^\ell$ is finite, it is compact and condition (i) is trivially fulfilled.
To prove condition (ii) we set $C:= ||Q||$ which is finite by assumption. Let $a_N=c_N(C+1)$ and let in the following $N$ be large enough so that $a_N<1$ and $||Q_N + R_N||\leq C +1.$  Now consider the discrete time Markov chain $(E^N_t, e^N_t)_{t \in \IN}$ which does not change its state with probability 
$1-a_N,$ changes only $e^N_t$ to $e^N_{t+1}= e^N_t+1$ with probability $a_N - \sum_{\tilde{\pi} \neq \pi} p^{(N)}_{\pi \tilde{\pi}}\geq 0$ if $E^N_{t+1}=E^N_t=\pi.$ Finally, we have $e^N_{t+1}= e^N_{t+1}$ and 
$E^N_{t+1}=\tilde{\pi}$ given that $E^N_t=\pi$ with probability $p^{(N)}_{\pi \tilde{\pi}}.$ This means that the process $E^N$ is just a version of $\Pi^{\ell,N}$ if the time between the steps is chosen to be $c_N.$ We have that the times between jumps of the Markov chain $(E^N,e^N)$ are given by i.i.d geometric  random variables $\tau_i^{N}$ with mean $\frac{1}{a_N}.$ Since $w(e^N_{[\frac{\cdot}{c^N}]}, \delta, T)=0$ whenever there exists a $J$ such that $\tau_i^N> \frac{\delta}{c^N}$ for $i=1,\dots, J$ and $\sum_{i=1}^J \tau_i^N\geq \frac{T}{c^N}$ it suffices to show that we can find $J$ and $\delta$ such that 
$$\liminf_{N \rightarrow \infty} \IP(\tau_i^N> \frac{\delta}{c^N} \text{ for }i=1,\dots, J \text{ and }\sum_{i=1}^J \tau_i^N\geq \frac{T}{c^N}) \geq 1-\eta.$$ This can be achieved since
\begin{eqnarray*}
 & &\IP(\tau_i^N> \frac{\delta}{c^N} \text{ for }i=1,\dots, J \text{ and }\sum_{i=1}^J \tau_i^N\geq \frac{T}{c^N}) \\
 &=&  \IP(\sum_{i=1}^J \tau_i^N\geq \frac{T}{c^N} | \tau_i^N> \frac{ \delta}{c^N} \text{ for }i=1,\dots, J )
 \cdot \IP(\tau_i^N> \frac{ \delta}{c^N} \text{ for }i=1,\dots, J)\\
 &\geq& \IP( \sum_{i=1}^J \tau_i^N\geq \frac{T }{c^N})
 \cdot \IP(\tau_i^N> \frac{ \delta}{c^N} )^J \rightarrow \IP(Z<J) \cdot \IP(X>\delta)^J,
\end{eqnarray*}
where $Z$ is a Poisson random variable with mean $T(C+1)$ and $X$ is exponential with mean $\frac{1}{C+1}.$ 
This finishes the proof.
\hfill $\Box$
\smallskip

%
%PROOF OF MAIN RESULT  Theorem \ref{thm:torusconv}
%
\subsection{Proof of Theorem \ref{thm:torusconv}}

In the following let  $\Pi^{\ell,L}$ be a spatial  $\Lambda$-$n$-coalescent on $T^L$ with transition kernel $p^L$.
We will write $\IP^{\pi}$ to indicate that $\Pi^{\ell,L}$ is started in $\pi^{\ell,L} \in \Ps^{\ell,L}_n$ at time $0.$
We want to show that asymptotically $\Pi^{\ell,L}$ behaves like a coalescing random walk for which only ever two blocks meet and coalesce instantaneously, and that the meetings only take place after the blocks have been "randomized" in space such that any two are equally likely to meet at any such event. In order to prove this we define another process of coalescing random walks $\tilde{\Pi}^{\ell,L}$
coupled to $\Pi^{\ell,L}$ (defined on the same probability space) and show that asymptotically $\Pi^{\ell,L}$ behaves like $\tilde{\Pi}^{\ell,L}$ 
and that $\tilde{\Pi}^{\ell,L}$ has the desired properties. Thus, the evolution of $\Pi^{L}$ (unlabeled) is given by a non-spatial Kingman coalescent as $L \rightarrow \infty.$

We first couple the spatial $\Lambda$-$n$-coalescent with a coalescing random walk.
Let $L \in \N$ and let $(\xi(i)^L)_{i \in [n]}$ be an i{.}i{.}d{.} family of random walks on $T^L$ with transition kernel $p^L$.
We now define a version of the spatial $\Lambda$-$n$-coalescent by stipulating that the label of a block $B$ follows the random walk $\xi^L(\min(B)$.

We now require some notation. Let $i,j \in [n]$ and $t \in \R_+$.
Let $A^L_t(i)$ define the unique block of $\Pi_t^{\ell,L}$ with $i \in A^L_t(i)$ and let $M^L_t(i) = \xi^L_t(\min(A_t^L(i)))$ be the location (label) of this block.
We set 
\begin{eqnarray}
\label{ijmeet}
\tau^L(i,j) &:=& \inf\set{t \geq 0 \Mid M^L_t(i) = M^L_t(j)},\\
\label{ijcoalesce}
\tau_c^L(i,j) &:=& \inf\set{t \geq 0 \Mid A^L_t(i) = A^L_t(j)}.
\end{eqnarray}
In words, these are the first times at which the two blocks containing $i$ and containing $j$ meet and coalesce, respectively. We also define the jump times
\begin{equation}
\label{meetcoalesce}
0=\tau_0^L< \tau_1^L < \dotsb  \quad \text{ and } \quad 0 = \tau_{c,0}^L <  \tau_{c,1}^L < \dotsb
\end{equation}
where 
\begin{equation*}
\tau_{k+1}^L:=\inf_{i,j \in [n]}\set{\tau^L(i,j)  \Mid \tau^L(i,j) > \tau_k^L}
\end{equation*}
 is the first meeting time after $\tau_k^L$ of blocks \emph{that have not met before} and $\tau_{c,k+1}^L$ defined analogously is the first time after $\tau_{c,k}^L$ at which blocks coalesce.

\noindent
Assuming that no two blocks of the starting configuration have the same label, we define the coalescing random walk $\tilde{\Pi}^{\ell,L}$ by first setting $\tilde{\Pi}^{\ell,L}_0=\Pi^{\ell,L}_0$. 
The label of a block $\tilde{B}$ in $\tilde{\Pi}^{\ell,L}$ is following the random walk $\xi^L(\min(\tilde{B}))$ and as soon as blocks have the same label they coalesce instantaneously.
We set $\tilde{M}^L$ and $\tilde{A}^L$ as before, now for the process $\tilde{\Pi}^{\ell,L}$.
We define the quantities in (\ref{ijmeet}) to (\ref{meetcoalesce}) analogously for the process $\tilde{\Pi}^{\ell,L}$ but note that $\tilde{\tau}^L(i,j)=\tilde{\tau}^L_c(i,j)$ and $\tilde{\tau}_k^L= \tilde{\tau}^L_{c,k}$.
Furthermore note that $\tilde{\tau}_k^L = \tau_k^L$ for $k=1$ but that  the analogous statement is not true for $k > 1$ since blocks meeting in the spatial 
$\Lambda$-$n$-coalescent could part without coalescing. 
We will see later though that $\tilde{\tau}_k^L = \tau_k^L$ becomes very likely for large $L$ if the blocks are initially sampled far apart.

We will now prove that the spatial $\Lambda$-$n$-coalescent behaves  after rescaling like the coupled coalescing random walk. We define for any random walk
$Z^L$ on $T^L$ with  transition kernel $p^L$  the stopping times  
	\begin{equation}
	\label{H^Ldef}
	H^L \defeq \inf \set{t\geq 0 \Mid Z_t^L = 0},
	\end{equation}
	which is the \emph{first hitting time of the origin} and 
	\begin{equation}
	W^L \defeq \inf \set{t \geq 0 \Mid Z_0^L = Z_t^L = 0 \text{ and there is an } s \in [0,t] \text{ with } Z_s^L \neq 0 },
	\end{equation}
	the  \emph{first return time to the origin}. We also use these definitions for $L=\infty$ with the convention that $T^{\infty}=\IZ^2.$

\begin{lemma}
	\label{lemma:meins1}
	Let $i,j \in [n].$ 
	Then a.s. 
	\begin{equation}
	\tau_c^L(i,j) - \tau^L(i,j) = \tau_{\text{coal}} + \sum_{i = 1}^{A-1} W^L_i,
	\end{equation}
	where  $W^L_1, W^L_2, \dotsc$ 
	are i.i.d copies of  $W^L$ of a random walk $Z^L$ on $T^L$ with transition kernel $2 p^L.$ 
	The random variables $\tau_{\text{coal}}$ and $A$ are independent from $\{W^L_i\}_{i \in \IN}$ and from each other and  $\tau_{\text{coal}} \sim \textup{exp}(\lambda_{2,2})$ 
	as well as $A \sim \textup{geom}\left(\lambda_{2,2} (2 + \lambda_{2,2})^{-1}\right).$ 
\end{lemma}

\begin{proof}
It is clear that $(M_t^L(i)-M_t^L(j))_{t \in \IR_x}$, the distance between the blocks containing $i$ and $j,$ is a rate $2$ random walk with transition kernel $p^L$ until the coalescence time $\tau_c^L(i,j).$ Using the strong Markov property we restart this random walk  at time $\tau^L(i,j)$, at which the blocks containing $i$ and $j$ meet for the first time. 
Therefore we restart in 0.
The time until one block migrates away from the other block is exponentially distributed with parameter $2,$ the time until they coalesce while in the same site is exponentially distributed with parameter $\lambda_{2,2}.$ Thus, the probability that the two coalesce before they part is given by  $\lambda_{2,2} (2 + \lambda_{2,2})^{-1}$ and this then happens after a time $\tau_{\text{coal}}.$ If they part it will take time $W^L_1$ to return to the same site. Using the strong Markov property we can repeat the argument leading to $A-1$ independent return times $W^L_i$ with $A$ as above, and thus to the statement of the lemma. 
\end{proof}

As an immediate corollary we obtain that on the time scale that we are considering (recall the definition of $s_L$ in (\ref{s_Ldef})) the time from the first encounter of two blocks until their eventual coalescence is asymptotically short.
\begin{cor}
	\label{kor:stopkonv}
	Let $i,j \in [n]$.	For all $\epsilon > 0$ and $\delta>0$ there is an $L(\epsilon,\delta)\in \N$, such that for all  $L \geq L(\epsilon,\delta)$ and for all $\pi^{\ell,L} \in \Ps^{\ell,L}_n$ we have
	\[\IP^{\pi^{\ell,L}}(s_L^{-1} \abs{\tau_c^L(i,j) - \tau^L(i,j)}\geq \epsilon) < \delta.\]
\end{cor}
\begin{proof}
	Let $L \in \N$ and $\pi^{\ell,L} \in \Ps^{\ell,L}_n$.
	If $A_0(i) = A_0(j)$ then $\tau_c^L(i,j) = 0 = \tau^L(i,j)$ and the claim follows. Now let $A_0(i) \neq A_0(j)$
	and let   $W^\infty_1, W^\infty_2, \dotsc$ be as in Lemma \ref{lemma:meins1} for $L = \infty$.
	If we consider the random walk on $T^L$ as the projection of a random walk on $\IZ^2$ that defines the $W^{\infty}_i$ then 
	it follows that $W^L_i \leq W^\infty_i$.
	With Lemma \ref{lemma:meins1} we have  for all $\pi^{\ell,L} \in \Ps^{\ell,L}_n$ that
	\begin{align*}
		\IP^{\pi^{\ell,L}}(s_L^{-1} \abs{\tau_c^L(i,j) - \tau^L(i,j)}\geq \epsilon) &= \IP^{\pi^{\ell,L}}\left(\tau_{\text{coal}} + \sum_{i = 1}^{A-1} W^L_i \geq s_L \epsilon \right) \\
		&\leq \IP^{\pi^{\ell,L}}\left(\tau_{\text{coal}} + \sum_{i = 1}^{A-1} W^\infty_i \geq s_L \epsilon \right).  
	\end{align*}
	Since the random walk on  $\Z^2$ is recurrent  we have $W^\infty_i < \infty$ almost surely. Likewise
	 $\tau_{\text{coal}}$ and  $A$ are almost surely finite. Since all these random variables  do not depend on  $L$ and ${\pi^{\ell,L}}$ the claim follows now with $s_L \konv[L] \infty.$ 
	\end{proof}

Let  ${\pi^{\ell,L}} \in \Ps^{\ell,L}_n$ and let every two blocks in ${\pi^{\ell,L}}$ have different labels.
It is clear, that in the coalescing random walk $\tilde{\Pi}^{\ell,L}$ starting in ${\pi^{\ell,L}}$ almost surely only pairwise mergers happen. 
So for $k \in [\#{\pi^{\ell,L}} -1]$ we have that $\tilde{\tau}^L_k< \infty$ is a time at which two blocks in $\tilde{\Pi}^{\ell,L}$  merge.
By swapping the random walks belonging to the two blocks merged at time $\tilde{\tau}^L_k$  after time $\tilde{\tau}^L_k$ we can define a new coalescing random walk which agrees with $\tilde{\Pi}^{\ell,L}$ until time $\tilde{\tau}^L_k$ and which has the same distribution as $\tilde{\Pi}^{\ell,L}$. More precisely, this means that if $ \tilde{B}_1$ and $ \tilde{B}_2$ are the two blocks that meet at time $\tilde{\tau}^L_k$ with $\min  \tilde{B}_1 < \min  \tilde{B}_2$ then the label of the newly created block follows the motion $\xi_t^L(\min \tilde{B}_2)$ rather than $\xi_t^L(\min  \tilde{B}_1).$ 
Let $\tilde{\tau}^{L,(2)}_{k+1}$ be the first meeting time of blocks in this process after time $\tilde{\tau}^L_k$.
For $q>0$ we consider the event
	\begin{equation}
	\label{B_Ldef}
	B^L \defeq B^L(q,{\pi^{\ell,L}}) \defeq  \{  \tilde{\tau}^L_{k+1} - \tilde{\tau}^L_k > s_L q, \tilde{\tau}^{L,(2)}_{k+1} - \tilde{\tau}^L_k > s_L q \text{ for all } k \in [\#{\pi^{\ell,L}} -2]\},
	\end{equation}
on which (after rescaling) it takes some time until a block that is created in a coalescence event meets further blocks whether it follows the motion of the first or the second block that took part in the coalescence event. Note that for $\#{\pi^{\ell,L}} = 2$ we have $B^L(q,{\pi^{\ell,L}})=\Omega$ since no conditions need to be met. 

%
%	THEOREM: COALESCENCE "ALMOST" INSTANTANEOUS
%
\begin{propn}
	\label{thm:konvergenz} Let $q>0.$
	For all $\epsilon>0$, $\delta > 0$ and $\nu > 0$ there are constants  $L(\epsilon,\delta,\nu) \in \N$ so that for all $L \geq L(\epsilon,\delta,\nu)$ 
	and all ${\pi^{\ell,L}} \in \Ps^{\ell,L}_n$ with $\#{\pi^{\ell,L}} \geq 2$, such that no two blocks of ${\pi^{\ell,L}}$ carry the same label, we have
	\begin{equation}
		\label{eq:Sprungzeit}
		\IP^{\pi^{\ell,L}}\left(s_L^{-1}\norm{(\tau_{c,1}^L, \dotsc, \tau_{c,n-1}^L) - (\tilde{\tau}_{1}^L, \dotsc, \tilde{\tau}_{n-1}^L)} \geq \epsilon, B^L(q,{\pi^{\ell,L}}) \right) < \delta 
	\end{equation}
	as well as 
	\begin{equation}
		\label{eq:Sprungart}
		\IP^{\pi^{\ell,L}}\left(\Pi^{L}_{\tau_{c,k}^L} \neq \tilde{\Pi}^{L}_{\tilde{\tau}_{k}^L} \text{ for some } k=1,\dots, \#{\pi^{\ell,L}}-1,B^L(q,{\pi^{\ell,L}})\right) < \nu.
	\end{equation}
\end{propn}
Note that the last statement (\ref{eq:Sprungart}) is about the unlabeled partition structures of the two processes.
%
%	PROOF THEOREM: COALESCENCE "ALMOST" INSTANTANEOUS 
%
\noindent
For proving this statement we set  for $\Pi^{\ell,L}$ started in ${\pi^{\ell,L}}$ 
\begin{equation}
\label{C^L(0)}
	C^L(0) \defeq \set{(i,j) \in [n]^2 \Mid i \nsim_{{\pi^{\ell,L}}} j}
	\end{equation}
	the set of index pairs that are initially in different blocks
\begin{proof} 
	For any $q>0$ there are almost surely exactly $\#{\pi^{\ell,L}}-1$ coalescence events of  $\tilde{\Pi}^{\ell,L}$ started in ${\pi^{\ell,L}}$ since at any coalescence event exactly two blocks merge.
	  We will first show the following statements by induction in  $k \in [\#{\pi^{\ell,L}}-1]:$ \\
		 There is an $L(\epsilon,\delta, \nu) \in \N$   such that for all ${\pi^{\ell,L}}$ and  $L \geq L(\epsilon,\delta,\nu),$
			\begin{eqnarray}
			 \label{enum:thm.1}
			&&\IP^{\pi^{\ell,L}}(\tau_{c,k}^L - \tilde{\tau}_{k}^L\geq s_L \epsilon, B^L(q,{\pi^{\ell,L}})) < \delta,\\
			 \label{enum:thm.1b}
			&&\IP^{\pi^{\ell,L}}(\tau_{c,k}^L < \tilde{\tau}_{k}^L, B^L(q,{\pi^{\ell,L}})) < \delta,\\
			\label{enum:thm.2}
			&&\IP^{\pi^{\ell,L}}(\tau_{c,k}^L \geq \tau_{k+1}^L, B^L(q,{\pi^{\ell,L}})) < \nu,\\
			\label{enum:thm.2b}
			&&\IP^{\pi^{\ell,L}}(\tau_{k+1}^L \neq  \tilde{\tau}_{k+1}^L, B^L(q,{\pi^{\ell,L}})) < \nu,\\
			\label{enum:thm.3}
			&&\IP^{\pi^{\ell,L}}(  \Pi^{\ell,L}_{\tau_{c,k}^L} \neq \tilde{\Pi}^{\ell,L}_{\tau_{c,k}^L}, B^L(q,{\pi^{\ell,L}})) < \nu.
			\end{eqnarray}
These statements will then complete the proof since (\ref{eq:Sprungzeit}) follows immediately from (\ref{enum:thm.1}) and (\ref{enum:thm.1b}) for all $k \in [\#{\pi^{\ell,L}}-1].$ Likewise, (\ref{eq:Sprungart}) follows from (\ref{enum:thm.1b}), (\ref{enum:thm.2}), and (\ref{enum:thm.3}) for all $k \in [\#{\pi^{\ell,L}}-1].$ 

	In order to start the induction let first $k=1$. We begin by showing (\ref{enum:thm.1}).
	Due to Corollary \ref{kor:stopkonv} we can find $L(\epsilon,\delta) \in \N$ with
	\[\IP^{\pi^{\ell,L}}(\tau_c^L(i,j) - \tau^L(i,j)\geq s_L \epsilon \text{ for all } (i,j) \in [n]^2) < \frac{\delta}{n^2}\]
	for all $L \geq L(\epsilon,\delta)$ and all ${\pi^{\ell,L}} \in \Ps^{\ell,L}_n$ with $\#{\pi^{\ell,L}} \geq 2.$
	Let $(i,j) \in C^L(0)$ so that  $\tau^L_c(i,j) \geq \tau^L_{c,1}$.  
	Thus, on the event $A_{i,j} \defeq \set{\tau_{1}^L = \tau^L(i,j)},$ it follows that
	\begin{align*}
		\IP^{\pi^{\ell,L}}(\tau_{c,1}^L - \tau_1^L\geq s_L \epsilon, A_{i,j},B^L(q,{\pi^{\ell,L}}) ) &= \IP^{\pi^{\ell,L}}(\tau_{c,1}^L - \tau^L(i,j) \geq s_L \epsilon, A_{i,j},B^L(q,{\pi^{\ell,L}}) ) \\
		&\leq \IP^{\pi^{\ell,L}}(\tau_c^L(i,j) - \tau^L(i,j)\geq s_L \epsilon ) < \frac{\delta}{n^2}.
	\end{align*}
	Hence, with $\IP^{\pi^{\ell,L}}(A_{i,j} \text{ for one } (i,j)\in C^L(0) ) = 1$ we have
	\[\IP^{\pi^{\ell,L}}(\tau_{c,1}^L - \tau_1^L\geq s_L \epsilon,B^L(q,{\pi^{\ell,L}}) ) \leq \sum_{(i,j)\in [n]^2} \IP^{\pi^{\ell,L}}(\tau_{c,1}^L - \tau_1^L\geq s_L \epsilon, A_{i,j},B^L(q,{\pi^{\ell,L}}) ) < \delta\]
	for all $L \geq L(\epsilon,\delta)$ and all ${\pi^{\ell,L}} \in \Ps^{\ell,L}_n$ with $\#{\pi^{\ell,L}} \geq 2,$ which is (\ref{enum:thm.1}) since $\tilde{\tau}_{1}^L=\tau_1^L$.

	Again because of $\tilde{\tau}_{1}^L=\tau_1^L$ statement (\ref{enum:thm.1b}) is trivial for $k=1$ since $\tau_{c,1}^L \geq \tau^L_1$ by definition.
	For showing (\ref{enum:thm.2}) assume that $\#{\pi^{\ell,L}} \geq 3$ as the statement is immediate for $\#{\pi^{\ell,L}}=2$.
	Observe that we have $\tau_2^L = \tilde{\tau}_2^L$ 	or $\tau_2^L = \tilde{\tau}_{2}^{L,(2)}$. In both cases the arguments are completely analogous so it suffices to show the result for $\tilde{\tau}_2^L$ 
	instead of $\tau_2^L ,$ and hence to consider the event $\tau^L_{c,1} \geq \tilde{\tau}_2^L.$

	We have $\tau^L_1=\tilde{\tau}^L_1 < \tilde{\tau}^L_{2} <\infty$ almost surely.
	From  (\ref{enum:thm.1}) for $k = 1$ and $\epsilon = q$ we obtain that there is an $L(\nu) = L(q,\nu) \in \N$ so that for all $L \geq L(\nu)$ and ${\pi^{\ell,L}} \in \Ps^{\ell,L}_n$ with $\#{\pi^{\ell,L}} \geq 3,$ 
	\begin{align*}
		\IP^{\pi^{\ell,L}}\left(\tau^L_{c,1} \geq \tilde{\tau}_2^L,B^L(q,{\pi^{\ell,L}})\right) &= \IP^{\pi^{\ell,L}}\left( \tau^L_{c,1} -  \tilde{\tau}_1^L \geq \tilde{\tau}_2^L - \tilde{\tau}_1^L \geq s_L q,B^L(q,{\pi^{\ell,L}})  \right)\\
				& \leq \IP^{\pi^{\ell,L}}\left( \tau^L_{c,1} -  \tilde{\tau}_1^L \geq s_L q,B^L(q,{\pi^{\ell,L}})  \right)  < \nu,
	\end{align*}
which proves  (\ref{enum:thm.2}).

	For showing (\ref{enum:thm.2b}) and (\ref{enum:thm.3}) assume again that $\#{\pi^{\ell,L}} \geq 3$ as the statement is trivially true for $\#{\pi^{\ell,L}}=2.$  We observe that from $\tau^L_{c,1} < \tau_2^L$ it follows that there is a coalescence event for $\Pi^{\ell,L}$ before any additional blocks meet.
	Thus, exactly the two blocks that met at time $\tau_1^L$ will have coalesced at time $\tau_{c,1}^L$ before meeting other blocks implying $ \Pi^{\ell,L}_{\tau_{c,1}^L} = \tilde{\Pi}^{\ell,L}_{\tau_{c,1}^L}$ and $\tau^L_{2} = \tilde{\tau}^L_2$. Thus,  (\ref{enum:thm.2b}) and (\ref{enum:thm.3}) follow from (\ref{enum:thm.2}).

	Let the claim now be true for $k \in [\#{\pi^{\ell,L}}-2]$ and let us consider $k+1 \in [\#{\pi^{\ell,L}}-1]$.
	We can assume $\tau^L_{c,k} < \tilde{\tau}^L_{k+1} = \tau^L_{k+1}$ and $\Pi^{\ell,L}_{\tau_{c,k}^L} = \tilde{\Pi}^{\ell,L}_{\tau_{c,k}^L}$ since these events have by the induction assumption asymptotically probability 1. 
	 Using the strong Markov property of $\Pi^{\ell,L}$ and $\tilde{\Pi}^{\ell,L}$ we can restart both processes at time $\tau_{c,k}^L$ (note that for $k+1$ all statements only concern the processes at times greater than $\tau_{c,k}^L$). 
	Since on the above events  both processes start in the same partition (for which different blocks have different labels)  the proof for $k+1$ works now analogous to the case $k=1$.
	The randomness of the new starting point poses no problem since we have uniform results (or alternatively because the space $\Ps^{\ell,L}_n$ is finite).
\end{proof}

%
%SUBSECTION: PROPERTIES OF COALESCING RANDOM WALKS
%

In order to study the asymptotic behavior of  $\tilde{\Pi}^{\ell,L}$ and also of $\Pi^{\ell,L}$ we need some results on coalescing random walks that were already shown in 
Cox (1989) \nocite{jC89} 
%
%
%THEOREM COX: HITTING TIME ORIGIN
%
\begin{propn}
	\label{thm:coxthm4}
	Let $Z^L$ be a simple symmetric random walk on $T^L$.
	Let $(a_L)_{L \in \N}$ be a sequence of nonnegative numbers with 
	\[\limty[L] L^{-1} \sqrt{\log L}\; a_L = \infty, \quad (a_L)_{L \in \N} = o(L).\] 
	Let $H^L$ be the first hitting time of the origin as in (\ref{H^Ldef}). Then
	\[\limty[L] \sup_{\norm{z} \geq a_L, t \geq 0} \abs{\IP^z\left(\frac{H^L}{s_L} > t\right) - \exp\left(-\frac{1}{2}\boldsymbol{\pi} t\right)} = 0. \]
\end{propn}
This is the special case $d=2$ of Theorem 4 in Cox (1989)\nocite{jC89}.

%
%THEOREM COX: COALESCENCE TIME OF TWO BLOCKS
%
\begin{cor}
	\label{kor:vonthm4}
	Let $i\neq j \in [n]$.
	Then we have
	\[\limty[L] \sup_{\pi^{\ell,L} \in [[a_L,\sqrt{2} L]], i \nsim_{\pi^{\ell,L}} j , t \geq 0} \abs{\IP^{\pi^{\ell,L}}
	\left(\frac{\tau^L(i,j)}{s_L} > t\right) - \exp\left(-\boldsymbol{\pi} t\right)} = 0. \]
\end{cor}
\begin{proof}
	We know that 
	$M^L_t(i) - M^L_t(j)$ (here the subtraction is done with respect to the cyclic structure of the torus) is up to time $\tau^L(i,j)$ a simple symmetric random walk on $T^L$ with jump rate $2$  and transition probabilities given by
	 $p^L$ that is started in  $z \defeq M^L_0(i) - M^L_0(j) \in T^L$ with $\norm{z} \geq a_L$ due to ${\pi^{\ell,L}} \in [[a_L,\sqrt{2} L]].$
	  The time $\tau^L(i,j)$ is just the hitting time of the origin of this random walk and so the claim follows from Proposition  \ref{thm:coxthm4}.
	\end{proof}

%
%THEOREM COX: COALESCENCE TIME OF ARBITRARY NUMBER OF BLOCKS
%
A more general statement holds for an arbitrary number of blocks that perform a coalescing random walk.
\begin{propn}
	\label{thm:cox3.2}
	Let $\tilde{\tau}^L_1=\tau^L_1$ be the first meeting time of any blocks. 
	Then for any $n \in \IN,$
	\[ \limty[L] \sup_{{\pi^{\ell,L}} \in [[a_L,\sqrt{2} L]] \cap \Ps^{\ell,L}_n, t \geq 0}\abs{\IP^{{\pi^{\ell,L}}}\left(\tau^L_1 > s_L t \right) -  \exp\left(- \boldsymbol{\pi} \binom{\#{\pi^{\ell,L}}}{2} t\right) } = 0. \]
\end{propn}

For $\#{\pi^{\ell,L}}=2$  this is the statement of Corollary \ref{kor:vonthm4}, for $\#{\pi^{\ell,L}}>2$ see (3.2) in Cox (1989)\nocite{jC89}.

\begin{lemma}
	\label{lem:cox3.7und3.8}
	Let $n \in \IN$ and let $K^L_C \subseteq \Ps^{\ell,L}_n$  for  $C \subseteq [n]$ be those labeled partitions $\pi^{\ell,L},$
	for which the blocks $\{A_0^L(i)\}_{i \in C}$ are pairwise different.
	Let  $T> 0$ and $i,j,k,l \in [n]$ be pairwise disjoint. 
	Then we have the following statements:
	\begin{enumerate}
	\item[(i)] \label{enum:cox3.7} 
			For $ I^L_{{\pi^{\ell,L}}} \defeq 	
				\int_0^{T s_L} \IP^{\pi^{\ell,L}} \Bigl( \tilde{\tau}^L_1 = \tilde{\tau}^L(i,j) \in \diff u, r_L\bigl(\tilde{M}^L_u(i), \tilde{M}^L_u(k)\bigr) \leq a_L \Bigr)$
			we have
			\[\limty[L] \sup_{{\pi^{\ell,L}} \in [[a_L, \sqrt2 L]] \cap K^L_{\{i,j,k\}}} I^L_{{\pi^{\ell,L}}} = 0. \]
	\item[(ii)] \label{enum:cox3.8} For 
			$J^L_{{\pi^{\ell,L}}} 
				\defeq	\int_0^{T s_L} \IP^{\pi^{\ell,L}} \Bigl( \tilde{\tau}^L_1 = \tilde{\tau}^L(i,j) \in \diff u, r_L\bigl(\tilde{M}^L_u(k), \tilde{M}^L_u(l)\bigr) \leq a_L \Bigr)$
			we have
			\[\limty[L] \sup_{{\pi^{\ell,L}} \in [[a_L, \sqrt2 L]]\cap K^L_{ \{i,j,k,l\}}   } J^L_{{\pi^{\ell,L}}} = 0, \]
	\end{enumerate}
\end{lemma}
\begin{proof}
	Since the structure of the blocks is irrelevant for the statement it suffices to consider $\tilde{\Pi}_0^{\ell,L} ={\pi^{\ell,L}}= \set{\set{i}\Mid i \in [n]}.$ 
	The two statements for the coalescing random walk $\tilde{\Pi}^{\ell,L}$ are now (3{.}7) und (3{.}8) in Cox (1989)\nocite{jC89}.
\end{proof}

The next result states that the probability for the event $B^L(q,{\pi^{\ell,L}})$ of (\ref{B_Ldef}) tends to $1$ as $L$ is large and $q$ is small, provided the blocks are sampled far enough apart. 
\begin{propn}
	\label{prop:Bistwahrscheinlich}
	For any $\epsilon > 0$ there is a $q(\epsilon) > 0$ and $L(\epsilon) \in \N$ such that for all $L \geq L(\epsilon)$ and all ${\pi^{\ell,L}} \in [[a_L, \sqrt{2}L]]$ we have
	\[\IP^{\pi^{\ell,L}}\Bigl(B^L\bigl(q(\epsilon),{\pi^{\ell,L}}\bigr)\Bigr) \geq 1 - \epsilon.\]
\end{propn}

\begin{proof} 
	We will show the following statements by induction over $k \in [\#\pi^{\ell,L}-1]$ for any $\epsilon>0.$
	There is a $q(\epsilon) > 0$ and $L(\epsilon) \in \N$ such that for all $L \geq L(\epsilon)$ and all ${\pi^{\ell,L}} \in [[a_L, \sqrt{2}L]]$ with $\#{\pi^{\ell,L}} \geq k$ we have 
			\begin{equation}
			\label{enum:Bistwahrscheinlich2} \IP^{\pi^{\ell,L}}(\tilde{\Pi}^{\ell,L}_{\tilde{\tau}_k^L} \in [[a_L, \sqrt{2}L]]) \geq1- \epsilon,
			 \end{equation}			
			\begin{equation} 
			\label{enum:Bistwahrscheinlich4} \IP^{\pi^{\ell,L}}\bigl( \tilde{\tau}^L_{k} - \tilde{\tau}^L_{k-1} > s_L q(\epsilon)\bigr) \geq 1- \epsilon. 
			\end{equation}
	In words, these events say that at the time $\tilde{\tau}_k^L$, the blocks are again far apart with high probability and the probability of meeting times being very close together is small.
	The proposition follows then from (\ref{enum:Bistwahrscheinlich4}) for all $k \in [\#\pi^{\ell,L}-1]$ since $\tilde{\tau}^L_{k} - \tilde{\tau}^L_{k-1}$ and $\tilde{\tau}^{L,(2)}_{k} - \tilde{\tau}^L_{k-1}$ have the same distribution.

	Note that since in $\tilde{\Pi}^{\ell,L}$ almost surely only pairwise mergers occur, all of the $\tilde{\tau}^L_k$ above are actual jump times of the process.

	We start with $k=1$ and first show (\ref{enum:Bistwahrscheinlich2}).
	From Proposition \ref{thm:cox3.2} it follows that there is an $L(\epsilon) \in \N$ and a $T(\epsilon)>0$ such that 
		\[\IP^{\pi^{\ell,L}}\bigl(\tilde{\tau}^L_1 > s_L T(\epsilon)\bigr) < \frac\epsilon2\]
	for all $L \geq L(\epsilon)$ and all ${\pi^{\ell,L}} \in [[a_L, \sqrt2 L]]$ with $\#{\pi^{\ell,L}} \geq 2$. According to Lemma \ref{lem:cox3.7und3.8}, by choosing $L(\epsilon)$ larger if necessary,
	\begin{equation*}
	\IP^{\pi^{\ell,L}}(\tilde{\tau}^L_1 \in [0,s_L T(\epsilon)],  \tilde{\Pi}^{\ell,L}_{\tilde{\tau}^L_1} \notin [[a_L, \sqrt{2} L]]) < \frac\epsilon2\\
	\end{equation*}
	for all $L \geq L(\epsilon)$ and all ${\pi^{\ell,L}} \in [[a_L, \sqrt2 L]]$ with $\#{\pi^{\ell,L}} \geq 2$.
	 Thus, for those $L$ and ${\pi^{\ell,L}}$
	\begin{align*}
		&\IP^{\pi^{\ell,L}}\big(\tilde{\Pi}^{\ell,L}_{\tilde{\tau}_1^L} \notin [[a_L, \sqrt{2}L]]\big) \\
		&\leq \IP^{\pi^{\ell,L}}\bigl(\tilde{\tau}^L_1 \in [0,s_L T(\epsilon)],  \tilde{\Pi}^{\ell,L}_{\tilde{\tau}^L_1} \notin  [[a_L, \sqrt{2} L]]\bigr) + \IP^{\pi^{\ell,L}}\bigl(\tilde{\tau}^L_1 > s_L T(\epsilon)\bigr) < \epsilon.
	\end{align*}

	In order to show (\ref{enum:Bistwahrscheinlich4}) we note that by Corollary \ref{kor:vonthm4} we can choose 
	$q(\epsilon)>0$ small enough and make $L(\epsilon) \in \N$ larger if necessary such that for all $L \geq L(\epsilon)$ and all ${\pi^{\ell,L}} \in [[a_L, \sqrt2 L]]$ with $\#{\pi^{\ell,L}} \geq 2$ 
	\begin{equation}
	\label{choice_of_q(eps)}
	\abs{\IP^{\pi^{\ell,L}}\left( s_L^{-1} \tau^L(i,j)  \leq q(\epsilon)\right)} < \frac\epsilon{n^2}.
	\end{equation}
	With $C^L(0)$ as in (\ref{C^L(0)}) it follows that 
	\begin{align*} 
		\IP^{\pi^{\ell,L}}\bigl(\tilde{\tau}^L_{1} - \tilde{\tau}^L_{0} \leq s_L q(\epsilon)\bigr) &\leq \sum_{(i,j) \in C^L(0)} \IP^{\pi^{\ell,L}}\bigl(s_L^{-1} \tilde{\tau}^L_{1} = s_L^{-1} \tau^L(i,j) \leq q(\epsilon)\bigr)\\
			&\leq  \sum_{(i,j) \in C^L(0)} \IP^{\pi^{\ell,L}}\bigl(s_L^{-1} \tau^L(i,j) \leq q(\epsilon)\bigr) 
			< \sum_{(i,j) \in C^L(0)} \frac\epsilon{n^2} \leq \epsilon.
	\end{align*}
Let the claim now be true for $k \in [n-2]$. Since (\ref{enum:Bistwahrscheinlich2}) holds for $k$ we can assume 
\[\tilde{\Pi}^{\ell,L}_{\tilde{\tau}_k^L} \in [[a_L, \sqrt{2}L]].\]
Using the strong Markov property of $\tilde{\Pi}^{\ell,L}$ we can restart the process at time $\tilde{\tau}_k^L.$ Since the blocks are again well seperated the induction step now follows analogous to the $k=1$ proof.
Note that the random starting point does not pose a problem since we have uniform results (or alternatively because $\Ps_n^L$ is finite).
\end{proof}

We will need the following properties of a random walk on the torus, which states that asymptotically the random walk will be uniformly distributed.
\begin{propn}
	\label{prop:cox2.8}
	Let $Z^L$ be a simple symmetric random walk on $T^L$ with transition kernel $p^L$.
	Let $(t_L)_{L \in \N}$ be a sequence with  $\limty[L] t_L = \infty$.
	Then
	\[\limty[L] \sup_{t \geq t_L (2L+1)^2} \sup_{x \in T^L} (2L+1)^2 \abs{p_t^L(x,0) - (2L+1)^{-2}} = 0\]
\end{propn}
This is  (2{.}8) in Cox (1989)\nocite{jC89}.
From Proposition \ref{prop:cox2.8} we obtain the asymptotic exchangeability of the blocks of the coalescing random walk.
\begin{lemma}
	\label{lem:raeumlaust}
	Let $\pi_0 \in \Ps_n$ be a partition with $\#\pi_0 \geq 2$ and for $L \in \N$ large enough let $\pi^{\ell,L} \in [[a_L, \sqrt2 L]]$ be a labeled partition, whose partition structure equals $\pi_0,$ meaning that  $\pi^L=\pi_0.$
	Furthermore, let $\sigma$ 
	be a  permutation of $[\#\pi_0]$ and $\pi^{\sigma,\ell,L}$ the labeled partition that is obtained from $\pi^{\ell,L}$ by permuting the labels of the block with $\sigma.$ 
	We define for  $L \in \N$ 
	\[q_L \defeq (\log\log L)(2L+1)^2.\]
	Let $k \in [\#\pi_0 - 1]$, then there exists a sequence $(\delta_L)_{L \in \N}$, independent of $\pi^{\ell,L}$ and $\sigma,$ with $\delta_L \konv[L] 0$ and
	\[\abs{\IE^{\pi^{\ell,L}}\left(f\left(\tilde{\Pi}^{\ell,L}_{q_L}\right)\right) - \IE^{\pi^{\sigma,\ell, L}}\left(f\left(\tilde{\Pi}^{\ell,L}_{q_L}\right)\right)} < \delta_L,\]
	for any measurable and bounded $f \colon \Ps_n^{\ell,L} \to \R.$
\end{lemma}
\begin{proof}
	Let $m \defeq \#\pi_0$ and $\pi^{\ell,L} =\{(B_1,\zeta_1), \dotsc, (B_m, \zeta_m), \dots \}.$
	We can couple the coalescing random walks started in  $\pi^{\ell,L}$ and $\pi^{\sigma, \ell,L}$ in a natural way by using the same motions for the corresponding blocks that start with the same labels.
	Thus, all times $\tilde{\tau}_k^L=\tilde{\tau}_{c,k}^L$ are identical for $k=1, \dots, m-1.$ 
	Since $s_L^{-1} q_L \rightarrow 0$ as $L \rightarrow \infty$ we have by Proposition \ref{thm:cox3.2} that 
	\begin{align}
	\label{q<tau}
		\IP^{\pi^{\ell,L}}(q_L < \tilde{\tau}^L_{1})=\IP^{\pi^{\sigma, \ell,L}}(q_L < \tilde{\tau}^L_{1}) \konv[L] 1
	\end{align}
	uniformly over all  $\pi^{\ell,L}, \pi^{\sigma,\ell,  L} \in [[a_L, \sqrt2 L]].$ Thus, due to the boundedness of $f$ it suffices to show that claim on the event
	$\set{q_L < \tilde{\tau}^L_1< \infty}.$
	For any  $z = (z_1, \dotsc,z_m) \in \left(T^L\right)^m$ we set $\pi_{z}=((B_1, z_1), \dots, (B_m, z_m), (\emptyset, \partial),\dots)$.
	While $\set{t \leq q_L < \tau^L_1}$ the blocks perform independent random walks with transition kernel $p^L.$ 
	Thus,  we have due to (\ref{q<tau}) that
	\begin{align*}	
		\IE^{\pi^{\ell,L}}\left(f\left(\tilde{\Pi}^{\ell, L}_{q_L}\right)\right) &= 
		\sum_{z \in \left(T^L\right)^m} f(\pi_{z}) \IP^{\pi^{\ell,L}}\left(\tilde{\Pi}^{\ell,L}_{q_L} = \pi_{z} \right) + o(1)
		= \sum_{z \in \left(T^L\right)^m} f(\pi_{z}) \prod_{k=1}^m p^L_{q_L}(\zeta_k-z_k,0) + o(1),
	\end{align*}
	where the $o(1)$ term converges to $0$ uniformly for $L \rightarrow \infty.$ We also have the analogous statement for $\pi^{\sigma,\ell, L}$ with $\sigma(\zeta_k)$ instead of $\zeta_k$. 
	Hence, 
	\begin{align*}
		&\abs{\IE^{\pi^{\ell,L}}\left(f\left(\tilde{\Pi}^{\ell,L}_{q_L}\right)\right) 
		- \IE^{\pi^{\sigma, \ell,L}}\left(f\left(\tilde{\Pi}^{\ell,L}_{q_L}\right)\right)}\\
		\quad&\leq \sum_{z \in \left(T^L\right)^m} \abs{f(\pi_{z})}\abs{  \left(\prod_{k=1}^m p^L_{q_L}\bigl(\zeta_k- z_k, 0 \bigr) - \prod_{k=1}^m p^L_{q_L}\bigl(\sigma(\zeta_k)- z_k, 0 \Bigr)\right)} + o(1)\\
		\quad&\leq \sum_{z \in \left(T^L\right)^m} \abs{f(\pi_{z})  \left(\prod_{k=1}^m p^L_{q_L}\bigl(\zeta_k- z_k, 0 \bigr) \right)- 
		(2L+1)^{-2m}}\\
		\quad&\quad+ \sum_{z \in \left(T^L\right)^m} \abs{f(\pi_{z})}\abs{  \left( \prod_{k=1}^m p^L_{q_L}\bigl(\sigma(\zeta_k)- z_k, 0 \bigr) \right)- (2L+1)^{-2m}} + o(1)\\
		\quad&\leq 2 \norm{f}_\infty \sup_{z \in (T^L)^m} \abs{\left(\prod_{k=1}^m (2L+1)^{2}p_{q_L}^L(z_k,0)\right) - 1} +o(1),
	\end{align*}
	where we have used in the last inequality that there are $(2L+1)^{2m}$ terms in the sums of the previous line. 
	With Proposition \ref{prop:cox2.8} (set $t_L \defeq \log\log L$)we have $(2L+1)^2 p_{q_L}^L(z,0) \to 1$ uniformly in $z$ and so the claim now follows since the right hand side converges uniformly over all $\pi^{\ell,L}$ and $\pi^{\sigma,\ell,L}.$ 
\end{proof}

	We now introduce some notation for the waiting time between meeting and coalescence times of blocks. Namely, for 
	$\pi^{\ell, L} \in \Ps_n^{\ell,L}$ we  define for $k \in [\#\pi^L-1]$ the \emph{waiting times}
	\begin{equation} \label{def:wartezeiten}
	\sigma^L_k \defeq \tau^L_k - \tau^L_{k-1}, \quad \sigma^L_{c,k} \defeq \tau^L_{c,k} - \tau^L_{c,k-1}.
	\end{equation}
	(Formally, set $\infty - \infty=0.$) 
	Analogously we define the waiting times $\tilde{\sigma}^L_k$ of $\tilde{\Pi}^{\ell,L}$.
	\label{def:kingmanspruenge}
	For $\pi_0 \in \Ps_n$ let $\bigl(K^{\pi_0}_t\bigr)_{t \in \R_+}$ be the non-spatial Kingman coalescent started in   $\pi_0$ and 
	 let $\tau_{K,k}$ be its $k$-th coalescence time for $k \in [\#\pi_0-1].$ 
	We set $\tau_{K,0} \defeq 0$ and define for $k \in [\#\pi - 1]$ the waiting times
	\begin{equation}
	U_k \defeq \tau_{K,k} - \tau_{K,k-1}.
	\end{equation}
	Note that the family of random variables $\set{U_k}_{ k \in [\#\pi_0-1]}$ is independent, and that
	\[U_k \sim \textup{Exp}\left(\binom{\#\pi_0 - (k-1)}{2}\right).\]
	Note also that they are independent of $\{K^{\pi_0}_{\tau_{K,k}}\}_{ k \in [\#\pi_0-1]}$
	as well as that $K^{\pi_0}_{\tau_{K,k}}$ results from $K^{\pi_0}_{ \tau_{K,k-1}}$ by coalescence of two blocks chosen at random.

We will now be able to prove a result about the  asymptotic behavior of the coalescence times and the types of transitions for coalescing random walks. Note that the result refers only to the partition structure $\tilde{\Pi}^{L}$
of the coalescing random walk and not to the labeled partitions $\tilde{\Pi}^{\ell,L}.$
\begin{thm}
	\label{thm:konvspruenge}
	Let $\pi_0 \in \Ps_n$ with $m:=\#\pi_0 \geq 2$.
	We consider for every  $L \in \N$ large enough a labeled partition $\pi^{\ell,L} \in [[a_L , \sqrt2 L]]$ with partition structure $\pi^{L}=\pi_0$.
	We then obtain convergence in distribution for $L \to \infty,$
	\begin{align*}
		&\left(\frac{\tilde{\sigma}^L_1}{s_L}, \tilde{\Pi}^{L}_{\tilde{\tau}^L_1}, \dotsc, \frac{\tilde{\sigma}^L_{m - 1}}{s_L}, \tilde{\Pi}^{L}_{\tilde{\tau}^L_{m-1}}\right) 
		\konvD 
		\bigl(\boldsymbol{\pi} U_1,K^{\pi_0}_{\tau_{K,1}}, \dotsc, \boldsymbol{\pi} U_{m - 1}, K^{\pi_0}_{\tau_{K,m - 1 }}\bigr).
	\end{align*}
\end{thm}
\begin{proof}
	The proof will follow along the lines of the analogous result in Limic and Sturm (2006)\nocite{LS06}  for $(U_1,\dotsc, U_{m-1})$ in dimension $d \geq 3$ and also use the results of Cox (1989)\nocite{jC89} for coalescing random walks. 
	We first show that there is a sequence $(\epsilon_L)_{L \in \N}$ with $\epsilon_L \konv[L] 0$ such that for all $k \in [m-1]$ 
		and $\pi^{\ell,L} \in [[a_L, \sqrt2 L]]$  as well as for all $u \geq 0,$ 
	\begin{equation}
		\label{eq:konvU}
		\IE^{\pi^{\ell,L}}\left( \abs{\IP^{\tilde{\Pi}^{\ell,L}_{\tilde{\tau}^L_{k}}}\left(\tilde{\tau}^L_1 >s_L  u \right) - \exp\left(- \pi \binom{\#\pi - k}{2} u \right)}\right) < \epsilon_L.
	\end{equation}
	Note that since no random walk jumps happen at the same time and since coalescence is instantaneous, no more than two blocks may meet and coalesce, so that $\#\tilde{\Pi}^L_{\tilde{\tau}^L_{k}} = \#\pi-k$ for all $k \in [m-1].$ For those $k$ we also set
	\[D_k^{L} \defeq \set{\tilde{\Pi}^{\ell,L}_{\tilde{\tau}^L_{k}} \in [[a_L,\sqrt2 L]]}.\] 
	Then, for all $u \geq 0$ we obtain 	\begin{eqnarray}
	\nonumber
		&&\IE^{\pi^{\ell,L}}\left( \abs{\IP^{\tilde{\Pi}^{\ell,L}_{\tilde{\tau}^L_{k}}}\left(\tilde{\tau}^L_1 >s_L  u \right) - \exp\left(- \pi \binom{\#\pi - k}{2} u \right)}\right)\\
		\label{EPPest}
			&\leq& \IE^{\pi^{\ell,L}}\left( 1_{D_k^L} \abs{\IP^{\tilde{\Pi}^{\ell,L}_{\tilde{\tau}^L_{k}}}\left(\tilde{\tau}^L_1 >s_L  u \right) - \exp\left(- \pi \binom{\#\pi - k}{2} u \right)}\right)
					+ 2\IP^{\pi^{\ell,L}}\left( (D_k^L)^C \right).
	\end{eqnarray}
	Proposition \ref{thm:cox3.2} and Lebesgue's dominated convergence theorem imply that the first term on the right hand side converges to zero uniformly in $\pi^{\ell,L}$ and $u.$ 
	The second  term on the right hand side of (\ref{EPPest})  converges to $0$ uniformly in $\pi^{\ell,L}$ and $u$ due Proposition \ref{prop:Bistwahrscheinlich} (see (\ref{enum:Bistwahrscheinlich2})
	  in the proof).This shows (\ref{eq:konvU}).

	We will now show convergence in distribution.	For
	%\[
	$t_1, \dotsc, t_{m - 1} \geq 0, \pi_{1}, \dotsc,\pi_{m - 1} \in \Ps_n$ %\] 
	and $k \in \left[m - 1\right]$  we define the event 
	\[A_k^L \defeq \set{\frac{\tilde{\sigma}^L_k}{s_L} > t_{k}, \tilde{\Pi}^{L}_{\tilde{\tau}^L_k} = \pi_k, \dotsc, \frac{\tilde{\sigma}^L_1}{s_L}> t_1, \tilde{\Pi}^{L}_{\tilde{\tau}^L_1} = \pi_1}.\]
	 Also set $A_0^L:=\Omega.$ 
	 Since only binary mergers are possible for coalescing random walks as well as the Kingman coalescent, 
	 it suffices to consider $\pi_k$ chosen such that $\pi_k$ results from $\pi_{k-1}$ through a coalescence of exactly two blocks
	for all $k \in [m-1].$ 
		We define
	\[C_k \defeq \set{(\min B_1, \min B_2) \Mid B_1, B_2 \text{ are blocks of } \pi_k, \min B_1< \min B_2}.\]
	Thus, we have
	$|C_k| = \binom{\#\pi_{k}}{2}=\binom{m-k}{2}.$ 
	Let $(i,j) \in C_{k-1}\setminus C_k$ meaning that on the event $A_k^L$ at time $\tilde{\tau}_k^L$ the blocks containing $i$ and $j$ in the partition $\pi_{k-1}$ merge to form the partition $\pi_k.$
	Thus, on the event $A^L_{k-1}$ we have $\{ \tilde{\Pi}^{L}_{\tilde{\tau}^L_k} = \pi_k\} = \{\tilde{\tau}^L_k= \tilde{\tau}^L(i,j)\}.$
	We set $q_L \defeq (\log\log L )(2L+1)^2$ for all $L \in \N$ and assume that we consider an $L$ large enough so that $q_L< s_L t_k$ for $k \in [m-1].$
Let $\tilde{\F}^L_{t}$ be the $\sigma$-algebra generated by $\tilde{\Pi}^{\ell,L}$ up to time $t.$ Then  $A^L_{k-1}$ is $\tilde{\F}_{\tilde{\tau}^L_{k-1}}$ measurable 
and we have that 
	\begin{align}
	\nonumber
		\IP^{\pi^{\ell,L}}(A_k^L) &= \IE^{\pi^{\ell,L}}\bigl(\IE^{\pi^{\ell,L}}(1_{\{ \tilde{\sigma}^L_k > s_L t_k, \tilde{\Pi}^{L}_{\tilde{\tau}^{L}_k} = \pi_k\}} \cdot 1_{A^L_{k-1}} | \tilde{\F}^L_{\tilde{\tau}^L_{k-1}})\bigr)\\
		\nonumber
			&= \IE^{\pi^{\ell,L}}\left(\IP^{\pi^{\ell,L}}\left(\tilde{\sigma}^L_k  >  s_L t_k,\tilde{\tau}^L_k= \tilde{\tau}^L(i,j) \middle| \tilde{\F}^L_{\tilde{\tau}^L_{k-1}}\right) 1_{A^L_{k-1}}\right)\\
			\label{PA_k^L}
			&= \IE^{\pi^{\ell,L}}\left(\IP^{ \tilde{\Pi}^{\ell,L}_{\tilde{\tau}^L_{k-1}}}\left(\tilde{\tau}^L_1 = \tilde{\tau}^L(i,j)>s_L  t_k \right) 1_{A^L_{k-1}}\right).
	\end{align}
By using that $\{\tilde{\tau}^L_1 >  s_L t_k\}=\{\tilde{\tau}^L_1 > q_L, \tilde{\tau}^L_1 -q_L> s_L t_k-q_L\}$ we obtain from conditioning on the information up to time $q_L$ that 
\begin{eqnarray}
\label{equalEf}
\IP^{ \tilde{\Pi}^{\ell,L}_{\tilde{\tau}^L_{k-1}}}\left( \tilde{\tau}^L_1 = \tilde{\tau}^L(i,j) > s_L t_k \right) 
= \IE^{ \tilde{\Pi}^{\ell, L}_{\tilde{\tau}^L_{k-1}}}\left(f\left(\tilde{\Pi}^{\ell,L}_{q_L}\right)  1_{\{\tilde{\tau}^L_1 > q_L \}}\right),
\end{eqnarray}
	where the function $f$ is defined by 
	\[f \colon \Ps^{\ell,L}_n \to \R, \quad \pi \mapsto \IP^\pi\left(\tilde{\tau}^L_1= \tilde{\tau}^L(i,j) > s_L t_k-q_L \right). \]
	Note that $f$ is measurable and bounded. For $(l,m) \in C_{k-1}$ let $\sigma_{l,m}$ be the permutation of $[\#\pi_{k-1}]$ that only swaps $i$ with $l$ and also $j$ with $m.$ Let  $\pi^{\sigma_{l,m},\ell,L}$ be the partition obtained from $\pi^{\ell,L}$ by permuting the labels with $\sigma^{l,m}$ (as in 	Lemma \ref{lem:raeumlaust}). We also set $\hat{A}^L_{k} \defeq A^L_{k} \cap D^L_{k}.$ Then, due to (\ref{PA_k^L}) and (\ref{equalEf}),
	\begin{align*}
		&\left| \right. \IP^{\pi^{\ell,L}}(A_k^L) -  \IE^{\pi^{\ell,L}}\left(\sum_{(l,m) \in C_{k-1}} \frac1{|C_{k-1}|} \IE^{ \tilde{\Pi}^{\sigma_{l,m},\ell,L}_{\tilde{\tau}^L_{k-1}}}\left(   f\left(\tilde{\Pi}^{\ell,L}_{q_L}\right) 1_{\{\tilde{\tau}^L_1 > q_L \}}\right) 1_{A^L_{k-1}}\right)\left. \right|\\
		\leq & \IE^{\pi^{\ell,L}}\left(\sum_{(l,m) \in C_{k-1}} \frac1{|C_{k-1}|}\left| \IE^{\tilde{\Pi}^{\ell,L}_{\tilde{\tau}^L_{k-1}}}\Bigl(  f\bigl(\tilde{\Pi}^{\ell, L}_{ q_L}\bigr)  1_{\{\tilde{\tau}^L_1 > q_L \}}\Bigr)
		 - \IE^{\tilde{\Pi}^{\sigma_{l,m},\ell,L}_{\tilde{\tau}^L_{k-1}}}\Bigl( f\bigl(\tilde{\Pi}^{\ell,L}_{q_L}\bigr) 1_{\{\tilde{\tau}^L_1 > q_L \}} \Bigr)\right| 1_{\hat{A}^L_{k-1}}\right)
		   +2 \IP^{\pi^{\ell,L}}((D^L_{k-1})^C).
	\end{align*}
	Since due to (\ref{eq:konvU}) the probability for the event $\{\tilde{\tau}^L_1 > q_L \}$ is arbitrarily close to $1$ for large $L,$ we have that the first term on the right hand side 
	converges to $0$ as $L \rightarrow \infty$ because of Lemma \ref{lem:raeumlaust} and Lebesgue's dominated convergence theorem. 
	The second  term converges to $0$ due to Proposition \ref{prop:Bistwahrscheinlich}. 
		Thus, we obtain
	\begin{equation*}
	\IP^{\pi^{\ell,L}}(A_k^L) = |C_{k-1}|^{-1} \IE^{\pi^{\ell,L}}\left(\sum_{(l,m) \in C_{k-1}} 
		\IE^{\tilde{\Pi}^{\sigma_{l,m},\ell,L}_{\tilde{\tau}^L_{k-1}}}\left(1_{\{\tilde{\tau}^L_1=\tilde{\tau}^L(i,j)>s_L t_k   \}}\right) 1_{A^L_{k-1}}\right) + o(1).
	\end{equation*}	
Note that for the process started in $\tilde{\Pi}^{\sigma_{l,m},\ell,L}_{\tilde{\tau}^L_{k-1}}$ the event $\{\tilde{\tau}^L_1=\tilde{\tau}^L(i,j)>s_L t_k   \}$ is the same as the event 
$\{\tilde{\tau}^L_1=\tilde{\tau}^L(l,m)>s_L t_k   \}$ for the process started in $\tilde{\Pi}^{\ell,L}_{\tilde{\tau}^L_{k-1}}.$ Since the events
 $\set{ \tilde{\tau}^L(l,m) = \tilde{\tau}^L_1, (l,m) \in C_{k-1}}$ are a partition of the probability space 
it then follows that 
		\begin{align*}
		\IP^{\pi^{\ell,L}}(A_k^L) &= |C_{k-1}|^{-1} \IE^{\pi^{\ell,L}}\left(\IP^{\tilde{\Pi}^{\ell,L}_{\tilde{\tau}^L_{k-1}}}\left(\tilde{\tau}^L_1>s_L  t_k \right) 1_{A^L_{k-1}}\right) + o(1)\\
			&= |C_{k-1}|^{-1} \IE^{\pi^{\ell,L}}\left(\left(\IP^{\tilde{\Pi}^{\ell,L}_{\tilde{\tau}^L_{k-1}}}\left(\tilde{\tau}^L_1>s_L t_k \right) - \exp\left(- \boldsymbol{\pi} \binom{m - (k-1)}{2}  t_k\right) \right) 1_{A^L_{k-1}}\right)\\
				&+  |C_{k-1}|^{-1} \exp\left(- \boldsymbol{\pi} \binom{m - (k-1)}{2} t_k\right)\IP^{\pi^{\ell,L}}(A^L_{k-1}) + o(1).
	\end{align*}
	It now follows from (\ref{eq:konvU}) that
	\[\limty[L]\IP^{\pi^{\ell,L}}(A_k^L) =  \binom{m - (k-1)}{2}^{-1} \exp\left(- \boldsymbol{\pi} \binom{m - (k-1)}{2} t_k \right) \limty[L]\IP^{\pi^{\ell,L}}(A^L_{k-1}). \]
	By induction we then have that
	\begin{align*}
		\limty[L] \IP^{\pi^{\ell,L}}(A_k^L) &= \limty[L] 
		\prod_{i = 1}^k  \binom{m-(i-1)}{2}^{-1} \exp\left(- \boldsymbol{\pi} \binom{m - (i-1)}{2} t_i\right) \\
		&= \prod_{i = 1}^k \IP( K^{\pi_0}_{\tau_{K,i}} = \pi_i | K^{\pi_0}_{\tau_{K,i-1}} = \pi_{i-1} ) \cdot \IP(\boldsymbol{\pi} U_i > t_i). 
	\end{align*}
Since this is the desired quantity this finishes the proof of convergence in distribution.
\end{proof}

We now  formulate the corresponding result for  $\Lambda$-$n$-coalescents.
\begin{cor}	
	\label{kor:konvspruengec}
	Let $\pi_0 \in \Ps_n$ with $m=\#\pi_0.$
	We consider for  $L \in \N$ large enough a labeled partition $\pi^{\ell,L} \in [[a_L , \sqrt2 L]]$ with corresponding partition structure $\pi^L =\pi_0$.
	Then there is  convergence in distribution for $L \to \infty,$
	\begin{align*}
		\left(\frac{\sigma_{c,1}^L}{s_L}, \Pi^L_{\tau_{c,1}^L}, \dotsc, \frac{\sigma_{c,m - 1}^L}{s_L}, \Pi^L_{\tau_{c,m-1}^L}\right) 
		\konvD \bigl(\boldsymbol{\pi} U_1,K^{\pi_0}_{\tau_{K,1}}, \dotsc, \boldsymbol{\pi} U_{m - 1}, K^{\pi_0}_{\tau_{K,m - 1 }}\bigr).
	\end{align*}
\end{cor}
\begin{proof}
	For  $L \in \N$ we set 
	\begin{align*}
		V^L &\defeq \left(\frac{\tilde{\sigma}^L_1}{s_L}, \tilde{\Pi}^L_{\tilde{\tau}^L_1}, \dotsc, \frac{\tilde{\sigma}^L_{m - 1}}{s_L}, \tilde{\Pi}^L_{\tilde{\tau}^L_{m-1}}\right),\\
		V^L_c &\defeq \left(\frac{\sigma_{c,1}^L}{s_L}, \Pi^L_{\tau_{c,1}^L}, \dotsc, \frac{\sigma_{c,m - 1}^L}{s_L}, \Pi^L_{\tau_{c,m-1}^L}\right),\\
		V_K &\defeq \bigl(\boldsymbol{\pi} U_1,K^{\pi_0}_{ \tau_{K,1}}, \dotsc, \boldsymbol{\pi} U_{m - 1}, K^{\pi_0}_{ \tau_{K,m - 1 }}\bigr).
	\end{align*}
	Due to Propositions \ref{thm:konvergenz} and \ref{prop:Bistwahrscheinlich} we have that $ V^L_c - V^L $ converges to $0$ in probability.
	From Theorem \ref{thm:konvspruenge} we obtain that $V^L \konvD V_K$. Taken together this implies that $V^L_c \konvD V_K$ as required.
\end{proof}

\label{subsec:bewresultat}
We are finally ready to prove our main result, Theorem \ref{thm:torusconv}.
\begin{proof}
	For any fixed $n$ we prove convergence in the  Skorohod space $\textup{D}(\R_+,\Ps_n),$
	\[\bigl(\Pi^L_{s_L t}\bigr)_{t \in \R_+} \konvD \bigl(K^{\pi_0}_{\boldsymbol{\pi} t}\bigr)_{t \in \R_+},\]
	by first showing relative compactness and then weak convergence of the finite dimensional distributions,
	see Theorem 3.7.8 of Ethier and Kurtz (1986)\nocite{EK}. 
	We start with the relative compactness of $\bigl((\Pi_{s_L t}^L)_{t \in \R_+}\bigr)_{L \in \N}$.
	Since $\Ps_n$ is compact it suffices to show by Theorem 3.6.3 of Ethier and Kurtz (1986)\nocite{EK} that for all $T>0$ and $\epsilon>0$ there exists a $\delta>0$ such that 
	\begin{equation}
	\label{modcontconv}
	\limsup_{L \rightarrow \infty} \IP^{\pi^{\ell,L}}(w(\Pi^L_{s_L \cdot}, \delta, T) \geq \epsilon) \leq \epsilon.
	\end{equation}
	where $w(\Pi^L_{s_L \cdot}, \delta, T)$ is the $\delta$-modulus of continuity of $\Pi^L_{s_L \cdot}$ on $[0,T]$ (see also (\ref{modcont})). 
	Namely, it is the infimum over all partitions of the form 
	$0=t_0< t_1< \cdots < t_{k-1} < T \leq t_k$ such that $t_i-t_{i-1}> \delta$ for all $1 \leq i \leq k$ of the quantity
	\[ \max_i \sup_{r,t \in [t_{i-1}, t_i)} 1\vee d_n(\Pi^L_{s_L r},\Pi^L_{s_L t}).\] 
	We set $T>0$ and $ \delta > 0$ as well as 
	\[c_1 \defeq \max \set{\boldsymbol{\pi} \binom{\#\pi_0 - (k - 1)}{2} \Mid k \in [\#\pi_0 - 1]}.\]
	Due to Corollary \ref{kor:konvspruengec} there is an $L(\delta)$ such that for all $L \geq L(\delta)$ and all $k \in [\#\pi_0 - 1]$ we have
	\begin{align*}
		\IP^{\pi^{\ell,L}}\left(s_L^{-1} \sigma_{c,k}^L < 2 \delta\text{ for one } k \in [\#\pi_0 - 1]\right) 
			< &\delta + \IP^{\pi_0}\bigl( \boldsymbol{\pi} U_k \leq 2\delta \text{ for one } k \in [\#\pi_0 -1]\bigr) \\
			\leq &\delta + (\#\pi_0 - 1) (1 -\exp\left(- 2 c_1\delta\right) ) 
			\leq  (1 + 2c_1 (\#\pi_0 - 1)) \delta.
	\end{align*}
	Since  $(\Pi_{s_L t}^L)_{t \in \R_+}$ is constant on the intervals $[s_L^{-1} \tau_{c,k-1}^L, s_L^{-1} \tau_{c,k}^L)$ we have on the event  
	that $s_L^{-1} \sigma_{c,k}^L > 2 \delta$ for all $k \in [\#\pi_0 -1]$ that $w(\Pi^L_{s_L \cdot}, \delta, T)=0.$
	Hence, (\ref{modcontconv}) and so the relative compactness follows by choosing $\delta=\epsilon (1 + 2c_1 (\#\pi_0 - 1))^{-1}.$

	Lastly, we show the weak convergence of the finite dimensional distributions.
	Let $l \in \N$ and $t_1,\dotsc,t_l \in \R_+.$
	By definition 
	\begin{equation}
		\label{eq:resultatbeweis}
		\Pi_{s_L t}^L = \sum_{k = 1}^{\#\pi_0 - 1} 1_{\{s_L t \in [\tau^L_{c,k-1}, \tau^L_{c,k})\}} \Pi_{\tau^L_{c,k-1}}^L.
	\end{equation}
	We define for  $k \in [l]$ the function $f_k \colon (\R_+ \times \Ps_n)^{\#\pi_0 - 1} \to \Ps_n$ by
	\begin{align*}
				(r_1,\pi_1,\dotsc,r_{\#\pi_0 -1}, \pi_{\#\pi_0 -1}) \mapsto 
			\begin{cases} 
				\pi_0 &\text{, } t_k \in [0,r_1)\\ 
				\pi_i &\text{, } t_k \in \left[\sum_{j = 1}^{i} r_j, \sum_{j = 1}^{i+1} r_j\right), i < \#\pi_0-1\\ 
				\pi_{\#\pi_0-1} &\text{, } t_k \in \left[\sum_{j = 1}^{\#\pi_0 - 1} r_j , \infty\right). 
			 \end{cases}
	\end{align*}
	We define the function
	\[f \defeq (f_1,\dotsc,f_l) \colon (\R_+ \times \Ps_n)^{\#\pi_0 - 1} \to \bigl(\Ps_n\bigr)^l. \]
	Let $V^L_c$ and $V_K$ be defined as in the proof of  Corollary \ref{kor:konvspruengec}.
	Due to  (\ref{eq:resultatbeweis}) we have
	\[f\left(V^L_c\right) = \bigl(\Pi^L_{s_L t_1}, \dotsc, \Pi^L_{s_L t_l}\bigr),\]
	and likewise
	\[f\bigl(V_K\bigr) = \bigl(K^{\pi_0}_{\boldsymbol{\pi} t_1}, \dotsc, K^{\pi_0}_{\boldsymbol{\pi} t_l}\bigr). \]
	Since  $U_k$ are continuous random variables we note that the event that $V_K$ takes values in the discontinuity set of $f$ has probability $0.$ 	Thus, due to $V^L_c \konvD V_K$ from Corollary \ref{kor:konvspruengec} we obtain
	\[\bigl(\Pi^L_{s_L t_1}, \dotsc, \Pi^L_{s_L t_l}\bigr) = f(V^L_c) \konvD f(V_K) = \bigl(K^{\boldsymbol{\pi}_0}_{\pi t_1}, \dotsc, K^{\pi_0}_{\boldsymbol{\pi} t_l}\bigr).\]
	This finishes the proof.
\end{proof}

%% The Appendices part is started with the command \appendix;
%% appendix sections are then done as normal sections
%% \appendix

%% \section{}
%% \label{}

%% References
%%
%% Following citation commands can be used in the body text:
%% Usage of \nocite is as follows:
%%   \nocite{key}          ==>>  [#]
%%   \nocite[chap. 2]{key} ==>>  [#, chap. 2]
%%   \nocitet{key}         ==>>  Author [#]

%% References with bibTeX database:

\bibliographystyle{model2-names} %alpha model1-num-names.bst
\bibliography{../references}

%% Authors are advised to submit their bibtex database files. They are
%% requested to list a bibtex style file in the manuscript if they do
%% not want to use model1-num-names.bst.

%% References without bibTeX database:

% \begin{thebibliography}{00}

%% \bibitem must have the following form:
%%   \bibitem{key}...
%%

% \bibitem{}

% \end{thebibliography}

\end{document}